\documentclass[twoside,10pt]{article}
\usepackage{amsmath,amsfonts,amssymb,amscd,amsthm,amsbsy,bm,latexsym}
\usepackage{gensymb}
\usepackage[pdftex,bookmarksnumbered=true,bookmarksopen=true,          %bookmark
         colorlinks=true,pdfborder=001,citecolor=blue,
            linkcolor=red,anchorcolor=green,urlcolor=blue]{hyperref}

\numberwithin{equation}{section}
\begin{document}

\title{\bf Convergence to the Self-similar Solutions to the Homogeneous Boltzmann Equation} \vskip 0.2cm
\author{{\bf Yoshinori Morimoto}\\
Graduate School of Human and Environmental Studies, Kyoto University\\
Kyoto 606-8501, Japan\\
E-mail address: morimoto@math.h.kyoto-u.ac.jp\\[2mm]
{\bf Tong Yang}\\
Department of Mathematics, City University of Hong Kong\\
Kowloon Tang, Hong Kong, China\\
E-mail address: matyang@cityu.edu.hk\\[2mm]
{\bf Huijiang Zhao}\\
School of Mathematics and Statistics, Wuhan University\\ Wuhan 430072, China\\
E-mail address: hhjjzhao@whu.edu.cn}
\date{ }

\maketitle

\vskip 0.2cm
\arraycolsep1.5pt
\newtheorem{Lemma}{Lemma}[section]
\newtheorem{Theorem}{Theorem}[section]
\newtheorem{Definition}{Definition}[section]
\newtheorem{Proposition}{Proposition}[section]
\newtheorem{Remark}{Remark}[section]
\newtheorem{Corollary}{Corollary}[section]

\begin{abstract}
The Boltzmann H-theorem implies that the solution to the Boltzmann equation tends to an equilibrium, that is, a Maxwellian when time tends to infinity. This has been proved in varies settings when the initial energy is finite. However, when the initial energy is infinite, the time asymptotic state is no longer described by a Maxwellian, but a self-similar solution obtained by Bobylev-Cercignani. The purpose of this paper is to rigorously justify this for the spatially homogeneous problem with Maxwellian molecule type cross section without angular cutoff.\\


\end{abstract}

%\tableofcontents

\section{Introduction}
\setcounter{equation}{0}
Consider the homogeneous Boltzmann equation
\begin{equation}\label{Boltzmann-equation}
\partial_tf(t,v)=Q(f,f)(t,v),\quad v\in{\mathbb{R}}^3,\ \ t\in {\mathbb{R}^+}
\end{equation}
with initial data
\begin{equation}\label{Boltzmann-IC}
f(0,v)=f_0(v)\geq 0,\quad v\in {\mathbb{R}^3},
\end{equation}
where the non-negative unknown function $f(t,v)$ is the  distribution density
 function of particles with velocity $v\in{\mathbb{R}}^3$ at time $t\in {\mathbb{R}^+}$.
The right hand side of \eqref{Boltzmann-equation} is  the Boltzmann bilinear collision operator corresponding to the Maxwellian molecule type cross section
\begin{equation}\label{collision-term}
Q(g,f)(v)=\int\limits_{{\mathbb{R}}^3}\int\limits_{{\mathbb{S}}^2}\mathcal{B}\left(\frac{v-v_*}{|v-v_*|}\cdot \sigma\right)\Big(f(v')g(v'_*)-f(v)g(v_*)\Big)d\sigma dv_*.
\end{equation}
Here for $\sigma\in {\mathbb{S}}^2$
$$
v'=\frac{v+v_*}{2}+\frac{|v-v_*|}{2}\sigma,\quad v'_*=\frac{v+v_*}{2}-\frac{|v-v_*|}{2}\sigma,
$$
 from the conservation of momentum and energy,
$$
v'+v'_*=v+v_*,\quad  |v'|^2+|v'_*|^2=|v|^2+|v_*|^2.
$$

The Maxwellian molecule type cross section $\mathcal{B}(\tau)$ in \eqref{collision-term} is  a non-negative function depending only on the deviation angle $\theta=\cos^{-1}(\frac{v-v_*}{|v-v_*|}\cdot \sigma)$. As usual, $\theta$ is restricted to $0\leq\theta\leq\frac\pi 2$ by replacing $\mathcal{B}(\cos\theta)$ by its ``symmetrized'' version $[\mathcal{B}(\cos\theta)+\mathcal{B}(\pi-\cos\theta)]{\bf 1}_{0\leq\theta\leq \pi/2}$. Moreover, motivated by inverse power laws,
throughout this paper, we assume
\begin{equation}\label{cross-section}
\lim\limits_{\theta\to 0_+}\mathcal{B}(\cos\theta)\theta^{2+2s}=b_0
\end{equation}
for positive constants $s\in (0,1)$ and $b_0>0$.

As in \cite{Cannone-Karch-CPAM-2010, Cannone-Karch-KRM-2013, Gabetta-Toscani-Wennberg-JSP-1995, Morimoto-KRM-2012, Toscani-Villani-JSP-1999},  the Cauchy problem \eqref{Boltzmann-equation} and \eqref{Boltzmann-IC} is considered in the set of probability measures on $\mathbb{R}^3$. For presentation, we first  introduce some function spaces defined in the previous literatures. For $\alpha\in [0,2]$, $\mathcal{P}^\alpha(\mathbb{R}^3)$ denotes the probability density function $f$ on $\mathbb{R}^3$ such that
$$
\int\limits_{\mathbb{R}^3}|v|^\alpha f(v)dv<\infty,
$$
and moreover when $\alpha\geq 1$, it requires that
$$
\int\limits_{\mathbb{R}^3}v_j f(v)dv=0,\quad j=1,2,3.
$$
Following \cite{Cannone-Karch-CPAM-2010}, a characteristic function $\varphi(t,\xi)$ is the Fourier transform of  $f(t,v)\in\mathcal{P}^0(\mathbb{R}^3)$ with respect to $v$:
\begin{equation}\label{Fourier-transform}
\varphi(t,\xi)=\hat{f}(t,\xi)=\mathcal{F}(f)(t,\xi)=\int\limits_{\mathbb{R}^3}e^{-iv\cdot \xi}f(t,v)dv.
\end{equation}
For each $\alpha\in[0,2]$, set $\widetilde{\mathcal{P}}^\alpha(\mathbb{R}^3) ={\mathcal{F}}^{-1}\left({\mathcal{K}}^\alpha(\mathbb{R}^3)\right)$ with
$\mathcal{K}(\mathbb{R}^3)=\mathcal{F}(\mathcal{P}^0(\mathbb{R}^3))$ and
$$
{\mathcal{K}}^\alpha(\mathbb{R}^3)=\Big\{\varphi\in \mathcal{K}(\mathbb{R}^3):  \|\varphi-1\|_{{\mathcal{D}}^\alpha}<\infty\Big\}.
$$
Here the distance $\mathcal{D}^\alpha$ between
two suitable  functions $\varphi(\xi)$ and $\tilde{\varphi}(\xi)$ with $\alpha>0$ is defined by
$$
\left\|\varphi-\tilde{\varphi}\right\|_{\mathcal{D}^\alpha}\equiv\sup\limits_{0\not=\xi\in{\mathbb{R}}^3}\frac{\left|\varphi(\xi)-\tilde{\varphi}(\xi)\right|}{|\xi|^\alpha}.
$$
Then the set ${\mathcal{K}}^\alpha(\mathbb{R}^3)$ endowed with the distance
${\mathcal{D}}^\alpha$
is a complete metric space. It follows from Lemma 3.12 of \cite{Cannone-Karch-CPAM-2010}, that $\mathcal{K}^\alpha(\mathbb{R}^3)=\{1\}$ for all $\alpha>2$ and the embeddings $\{1\}\subset \mathcal{K}^\alpha(\mathbb{R}^3)\subset \mathcal{K}^\beta(\mathbb{R}^3)\subset \mathcal{K}(\mathbb{R}^3)$ hold for $2\geq \alpha\geq \beta\geq 0$.

The advantage of considering the Maxwellian molecule cross section
is that the Bobylev formula is in a simple form.
That is,  by taking the Fourier transform \eqref{Fourier-transform} of
the equation \eqref{Boltzmann-equation} leads to the following equation for the new unknown $\varphi=\varphi(t,\xi)$:
\begin{equation}\label{Reformulated-equation}
\partial_t\varphi(t,\xi)=\int\limits_{{\mathbb{S}^2}}\mathcal{B}\left(\frac{\xi\cdot\sigma}{|\xi|}\right)\left(\varphi\left(t,\xi^+\right)
\varphi\left(t,\xi^-\right)-\varphi(t,\xi)\right)d\sigma,
\end{equation}
where we have used
$$
\varphi(t,0)=\int\limits_{{\mathbb{R}}^3}f(t,v)dv=1.
$$
Here,
\begin{equation}\label{xi_pm}
\xi^+=\frac{\xi+|\xi|\sigma}{2},\quad \xi^-=\frac{\xi-|\xi|\sigma}{2}
\end{equation}
satisfying
\begin{equation}\label{relation-xi_pm}
\xi^++\xi^-=\xi,\quad |\xi^+|^2+|\xi^-|^2=|\xi|^2.
\end{equation}
From now on, we consider the Cauchy problem for \eqref{Reformulated-equation} with initial
condition
\begin{equation}\label{Reformulated-IC}
\varphi(0,\xi)=\varphi_0(\xi).
\end{equation}
For $\alpha\in (2s,2]$, it is shown in \cite{Cannone-Karch-CPAM-2010, Morimoto-KRM-2012, Morimoto-Wang-Yang-2014} that this Cauchy problem
admits a unique global solution $\varphi(t,\xi)\in C\left([0,\infty), \mathcal{K}^\alpha({\mathbb{R}^3})\right)$ for every $\varphi_0(\xi)\in\mathcal{K}^\alpha({\mathbb{R}^3})$. Moreover, $f(t,\cdot)\in L^1({\mathbb{R}^3}) \cap H^\infty(\mathbb{R}^3)$ for any $t>0$ if $\mathcal{F}^{-1}(\varphi_0)(v)$ is not a single Dirac mass, cf. \cite{Morimoto-Wang-Yang-2014, Morimoto-Yang-2012}.

To study the large time behavior of the solution, it depends on whether the
initial energy is finite or not, and in the above setting, it depends on
the parameter $\alpha$, cf. \cite{Arkeryd-CMP-1982, Cannone-Karch-CPAM-2010, Cannone-Karch-KRM-2013, Gabetta-Toscani-Wennberg-JSP-1995, Morimoto-KRM-2012, Morimoto-Wang-Yang-2014, Pulvirenti-Toscani-AMPA-1996, Tanaka-WVG-1978, Toscani-Villani-JSP-1999} and the references cited therein:
\begin{itemize}
\item When $\alpha=2$, the initial datum has finite energy so that
the solution tends to the Maxwellian defined by the initial datum.
This was indeed proved in the early work by Tanaka \cite{Tanaka-WVG-1978}
using probability theory
in the weak convergence in probability. And it was also proved later
 in \cite{Gabetta-Toscani-Wennberg-JSP-1995, Pulvirenti-Toscani-AMPA-1996,  Toscani-Villani-JSP-1999} by using analytic methods about convergence in Toscani metrics. %$\mathcal{K}^\gamma(\mathbb{R}^3)-$topology for %some $\gamma$ related to the parameter $\alpha$.
  Moreover, if some moment higher than the second
 order is assumed to be bounded, the convergence in the $\mathcal{D}^{2+\delta}-$distance with $\delta>0$ is shown to be exponentially decay in time, cf. \cite{Gabetta-Toscani-Wennberg-JSP-1995};
\item When $2s<\alpha<2$, the initial energy
is  infinite  so that the solution will no longer tend to an equilibrium, but
to a
self-similar solution
$$
f_{\alpha,K}(t,v)=e^{-3\mu_\alpha t}\Psi_{\alpha, K}\left(ve^{-\mu_\alpha t}\right)
$$
constructed in \cite{Bobylev-Cercignani-JSP-2002a, Bobylev-Cercignani-JSP-2002b},
where
\begin{equation}\label{mu-alpha}
\mu_\alpha=\frac{\lambda_\alpha}{\alpha}, \quad \lambda_\alpha\equiv\int\limits_{\mathbb{S}^2}\mathcal{B}\left(\frac{\xi\cdot\sigma}{|\xi|}\right)\left(\frac{|\xi^-|^\alpha+|\xi^+|^\alpha}{|\xi|^\alpha}-1\right)d\sigma.
\end{equation}
Here, $K>0$ is any given constant and $\Psi_{\alpha,K}(v)$ is a radially symmetric non-negative function satisfying
\begin{equation}\label{property-SS}
\int\limits_{{\mathbb{R}^3}}\Psi_{\alpha,K}(v)dv=1,\quad \hat{\Psi}_{\alpha,K}(\xi)\in {\mathcal{K}}^\alpha({\mathbb{R}^3}),\quad \lim\limits_{|\eta|\to 0}\frac{1-\hat{\Psi}_{\alpha,K}(\eta)}{|\eta|^\alpha}=K.
\end{equation}
The regularity of
the self-similar solution in $H^\infty({\mathbb{R}^3})$ was proved
 in \cite{ Morimoto-Wang-Yang-2014, Morimoto-Yang-2012}.
However,
the convergence   to the self-similar solution $f_{\alpha,K}(t,v)$ is not well understood even though there are some
works, cf. \cite{Bobylev-Cercignani-JSP-2002b, Cannone-Karch-CPAM-2010, Cannone-Karch-KRM-2013} about pointwise convergence in radially symmetric setting or in
weak topology with scaling. In fact, even how to show convergence in distribution
sense has been a problem.
\end{itemize}

The main difficulties in studying the convergence to the self-similar solutions
come from the fact that the self-similar solution has infinite energy and
it decays to zero exponentially in time except in $L_1$ norm. The purpose of
this paper is to show strong
 convergence holds when $\alpha\in(\max\{2s,1\}, 2]$
under some conditions on the initial perturbation.

For this, we first consider the $\mathcal{D}^{2+\delta}$ distance between
two solutions.  For $f_0(v)\in \widetilde{\mathcal{P}}^\alpha(\mathbb{R}^3)$
and $ g_0(v)\in \widetilde{\mathcal{P}}^\alpha(\mathbb{R}^3)$, as in \cite{Gabetta-Toscani-Wennberg-JSP-1995, Ikenberry-Truesdell-JRMA-1956}, set
\begin{eqnarray}\label{correction-funtion}
\widetilde{P}(t,\xi)&=&e^{-At}\widetilde{P}(0,\xi),\nonumber\\
\widetilde{P}(0,\xi)&=&\frac 12\sum\limits_{j,l=1}^3\xi_j\xi_l P_{jl}(0)X(\xi),\\
P_{jl}(0)&=&\int\limits_{\mathbb{R}^3}\left(v_jv_l-\frac{\delta_{jl}}{3}|v|^2\right)(f_0(v)-g_0(v))dv,\nonumber
\end{eqnarray}
where
\begin{equation}\label{A}
A=\frac 34\int\limits_{\mathbb{S}^2}\mathcal{B}\left(\frac{\sigma\cdot\xi}{|\xi|}\right)
\left(1-\left(\frac{\sigma\cdot\xi}{|\xi|}\right)^2\right)d\sigma,
\end{equation}
$\delta_{jl}$ is the Kronecker delta and $X(\xi)\equiv X(|\xi|)$ is a  smooth radially symmetric function satisfying $0\leq X(\xi)\leq 1$ and $X(\xi)=1$ for $|\xi|\leq 1$ and $X(\xi)=0$ for $|\xi|\geq 2$.

The first result in this paper on the $\mathcal{D}^{2+\delta}$
time asymptotic stability of the solutions is given by

\begin{Theorem}\label{Thm1.1} Suppose  $f_0(v), g_0(v)\in \widetilde{\mathcal{P}}^\alpha(\mathbb{R}^3)$
for $\alpha\in(\max\{2s,1\}, 2]$.  Let $\hat{f}(t,\xi)$ and $ \hat{g}(t,\xi)$ be the corresponding two global solutions of the Cauchy problem \eqref{Reformulated-equation} with initial data $\hat{f}_0(\xi)$ and $\hat{g}_0(\xi)$ respectively. Assume for some
 $\delta\in(0,\alpha]\cap\left(0, \frac{A}{\mu_\alpha}\right)$, the initial data satisfy
\begin{equation}\label{zero-2-perturbation}
\int\limits_{{\mathbb{R}^3}}|v|^2(f_0(v)-g_0(v))dv=0,
\end{equation}
\begin{equation}\label{high-integrability}
\left\{
\begin{array}{l}
{\displaystyle\int\limits_{{\mathbb{R}^3}}}|v|^{2}|f_0(v)-g_0(v)|dv<+\infty,\\[5mm]
\left\|\hat{f}_0(\cdot)-\hat{g}_0(\cdot)-\widetilde{P}(0,\cdot)\right\|_{\mathcal{D}^{2+\delta}}<+\infty.
\end{array}
\right.
\end{equation}
Then there exists some positive constant $C_1>0$ independent of $t$ and $\xi$ such that
\begin{equation}\label{stability-estimate}
\left\|\hat{f}(t,\cdot)-\hat{g}(t,\cdot)-\widetilde{P}(t,\cdot)\right\|_{\mathcal{D}^{2+\delta}}\leq C_1e^{-\eta_0t}, \quad t\geq 0.
\end{equation}
Here, $\eta_0=\min\left\{A-\delta\mu_{\alpha}, B\right\}$ and
\begin{equation}\label{B}
B=\int\limits_{\mathbb{S}^2}\mathcal{B}\left(\frac{\sigma\cdot\xi}{|\xi|}\right)\left(1-\left|\cos\frac\theta 2\right|^{2+\delta}-\left|\sin\frac\theta 2\right|^{2+\delta}\right)d\sigma,\quad \cos\theta=\frac{\sigma\cdot\xi}{|\xi|}.
\end{equation}
\end{Theorem}

Note that for the $\mathcal{D}^{2+\delta}$ convergence to the self-similar solution,
one can simply take $g_0=\Psi_{\alpha,K}(v)$. Based on this, in order to obtain
a convergence in strong topology, such as in the Sobolev norms, we will give
a uniform in time estimate on the solution in $H^N$-norm that is given in

\begin{Theorem}\label{Thm1.2} For $\max\left\{1,2s\right\}<\alpha<2$,
  assume that $f_0(v)\in  \widetilde{\mathcal{P}}^\alpha(\mathbb{R}^3)$ satisfies
\eqref{zero-2-perturbation}-\eqref{high-integrability} and
 is not a single Dirac mass,
$g_0(v)=\Psi_{\alpha,K}(v)$. Then for any given positive constant $t_1>0$ and any $N\in{\mathbb{N}}$,  there exists a positive constant $C_2(t_1, N)$ independent of
$t$ such that
\begin{equation}\label{uniform-H-N-bound}
\sup\limits_{t\in[t_1,+\infty)}\Big\{\left\|f(t,\cdot)\right\|_{H^N}\Big\}\leq C_2(t_1,N).
\end{equation}
Consequently, there exists a positive constant $C_3(t_1,N)$ independent of $t$ such that
\begin{equation}\label{H-N-decay}
\Big\|f(t,\cdot)-f_{\alpha, K}(t,\cdot)\Big\|_{H^N}=\Big\|f(t,v)-e^{-3\mu_\alpha t}\Psi_{\alpha, K}\left(ve^{-\mu_\alpha t}\right)\Big\|_{H^N}\leq C_3(t_1,N) e^{-\frac{\eta_0 t}{2}}
\end{equation}
holds for any $t\geq t_1$.
Since
\begin{equation}\label{decay-similarity}
e^{-\frac{3\mu_\alpha t}{2}}\Big\|\Psi_{\alpha,K}(\cdot)\Big\|_{L^2}\leq\Big\|f_{\alpha,K}(t,\cdot)\Big\|_{H^N}\leq e^{-\frac{3\mu_\alpha t}{2}}\Big\|\Psi_{\alpha,K}(\cdot)\Big\|_{H^N},
\end{equation}
 \eqref{H-N-decay} and \eqref{decay-similarity} imply that  when
\begin{equation}\label{mu-eta}
\mu_\alpha<\frac{\eta_0}{3},
\end{equation}
the convergence rate given in \eqref{H-N-decay} is faster than the decay rate of the self-similar solution itself. Hence in this case, the  infinite energy  solution $f(t,v)$ converges to the
self-similar solution $f_{\alpha,K}(t,v)$ exponentially in time.
\end{Theorem}
\begin{Remark} Since $\mu_\alpha\to 0+$ as $\alpha\to 2$,
the condition \eqref{mu-eta} holds when $\alpha$ is  close to $2$.
\end{Remark}

For the case with finite energy, the above stability estimates give a better
convergence description on the solution obtained in the previous literatures, which extends the exponential convergence result in the Toscani metrics $\mathcal{D}^{2+\delta}$ with $\delta>0$, cf. \cite{Gabetta-Toscani-Wennberg-JSP-1995}, to the Sobolev space $H^N(\mathbb{R}^3)$ for any $N\in{\mathbb{N}}$. In
fact, we have

\begin{Corollary}\label{decay-Maxwellian} Suppose that $f_0(v)\in  \mathcal{P}^2(\mathbb{R}^3)$   is not a single Dirac mass and satisfies
\begin{equation}\label{high-Maxwellian-integrability}
\int\limits_{\mathbb{R}^3}|v|^2f_0(v)dv=3,\quad
\left\|\hat{f}_0(\cdot)-\mu(\cdot)-\widetilde{P}(0,\cdot)\right\|_{\mathcal{D}^{2+\delta}}<+\infty,
\end{equation}
for some positive constant $\delta\in (0,2]$ with $\mu=(2\pi)^{-\frac 32}e^{-|v|^2/2}$.
Then for any $N\in{\mathbb{N}}$, there exist positive constants $C_4, C_5(t_1, N)>0$  independent of $t$  such that
\begin{equation}\label{stability-Maxwellian}
\left\|\hat{f}(t,\cdot)-\mu(\cdot)-\widetilde{P}(t,\cdot)\right\|_{\mathcal{D}^{2+\delta}}\leq C_4e^{-\eta_1t}, \quad t>0,
\end{equation}
and
\begin{equation}\label{uniform-bound}
\sup\limits_{t\in[t_1,+\infty)}\Big\{\left\|f(t,\cdot)\right\|_{H^N}\Big\}\leq C_5(t_1,N),\quad t\geq t_1.
\end{equation}
 Here $t_1>0$ is any given positive constant and $\eta_1=\min\left\{A, B\right\}$.

A direct consequence of \eqref{stability-Maxwellian} and \eqref{uniform-bound}
gives
\begin{equation}\label{decay-H-N-Maxwellian}
\Big\|f(t,\cdot)-\mu(\cdot)\Big\|_{H^N}\leq C_6(t_1,N) e^{-\frac{\eta_1t}{2}},
\end{equation}
for some positive constant $C_6(t_1,N)$ depending only on $t_1$ and $N$.
\end{Corollary}

\begin{Remark} Two comments on the above two theorems:
\begin{itemize}
\item By Lemma \ref{K-2+delta-bound},   sufficient conditions for the
requirements \eqref{high-integrability} and \eqref{high-Maxwellian-integrability} are
\begin{equation*}
\int\limits_{\mathbb{R}^3}|v|^{2+\delta}\left| f_0(v)-g_0(v)\right|dv<+\infty,
\end{equation*}
and
\begin{equation*}
\int\limits_{\mathbb{R}^3}|v|^{2+\delta}\left| f_0(v)-\mu(v)\right|dv<+\infty,
\end{equation*}
respectively.
\item The convergence rate in Corollary \ref{decay-Maxwellian} is faster than the corresponding rates  in Theorem \ref{Thm1.1} and Theorem \ref{Thm1.2}.
%\item The global solution $f(t,v)$ constructed in \cite{Villani-ARMA-1998} to the Cauchy problem \eqref{Boltzmann-equation}, \eqref{Boltzmann-IC} satisfies all the conditions  in Corollary \ref{decay-Maxwellian}.
\end{itemize}
\end{Remark}

Before the end of  the introduction, we list some notations used throughout the paper. Firstly, $C$, $C_i$ with $i\in\mathbb{N}$, and $O(1)$ are used for some generic large positive constants and $\varepsilon, \kappa$ stand for some generic small positive constants.  When the dependence needs to be explicitly pointed out, then the notations like
$C(\cdot,\cdot)$  are used. For multi-index $\beta=(\beta_1,\beta_2,\beta_3)$, $\partial^\beta_v=\partial^{\beta_1}_{v_1}\partial^{\beta_2}_{v_2}\partial^{\beta_3}_{v_3}$. And $A\lesssim B$ means that there is a constant $C>0$ such that $A\leq CB$, and $A\sim B$ means $A\lesssim B$ and $B\lesssim A$.

The rest of this paper will be organized as follows: Some known results concerning the global solvability, stability, and regularity of solutions to the Cauchy problem \eqref{Reformulated-equation} and
\eqref{Reformulated-IC} in $\mathcal{K}^\alpha(\mathbb{R}^3)$ are
recalled in Section 2. Moreover, some properties of the approximations of the initial data in $\mathcal{K}^\alpha(\mathbb{R}^3)$ will also be given in this section.
And then the proofs of Theorem \ref{Thm1.1}, Theorem \ref{Thm1.2}, and Corollary \ref{decay-Maxwellian} will be given in the next three sections respectively.

\section{Preliminaries}
\setcounter{equation}{0}

In this section, we wil first recall the global solvability, stability and regularity results on the Cauchy problem \eqref{Reformulated-equation} and \eqref{Reformulated-IC} obtained in \cite{Bobylev-Cercignani-JSP-2002a, Bobylev-Cercignani-JSP-2002b, Cannone-Karch-CPAM-2010, Cannone-Karch-KRM-2013, Morimoto-KRM-2012, Morimoto-Wang-Yang-2014, Morimoto-Yang-2012}. And then we will
study  the  properties of the approximation $f_{0R}(v)$ on the initial
data $f_{0}(v)$ defined in \eqref{approximation-f-0} for later stability estimates.

For the Cauchy problem \eqref{Reformulated-equation}
and
 \eqref{Reformulated-IC}, the following estimates are proved in \cite{Cannone-Karch-CPAM-2010, Cannone-Karch-KRM-2013, Morimoto-KRM-2012, Morimoto-Wang-Yang-2014, Morimoto-Yang-2012}.
\begin{Lemma}\label{global-solvability} For $\alpha\in(2s,2]$, if $\varphi_0(\xi)\in\mathcal{K}^\alpha(\mathbb{R}^3)$, then the Cauchy problem   \eqref{Reformulated-equation}
and
 \eqref{Reformulated-IC}
admits a unique global classical solution $\varphi(t,\xi)\equiv \hat{f}(t,\xi)\in C\left([0,\infty), \mathcal{K}^\alpha(\mathbb{R}^3)\right)$ satisfying
\begin{equation}\label{K-alpha-stability-1}
\Big\|\varphi(t,\cdot)-1\Big\|_{\mathcal{D}^\alpha}\leq e^{\lambda_\alpha t}
\Big\|\varphi_0(\cdot)-1\Big\|_{\mathcal{D}^\alpha}.
\end{equation}
If $ \psi(t,\xi)\in C\left([0,\infty), \mathcal{K}^\alpha(\mathbb{R}^3)\right)$
is another solution with initial data
 $\psi_0(\xi)\in \mathcal{K}^\alpha(\mathbb{R}^3)$,  then
\begin{equation}\label{K-alpha-stability}
\Big\|\varphi(t,\cdot)-\psi(t,\cdot)\Big\|_{\mathcal{D}^\alpha}\leq e^{\lambda_\alpha t}
\Big\|\varphi_0(\cdot)-\psi_0(\cdot)\Big\|_{\mathcal{D}^\alpha}.
\end{equation}
Furthermore, if $f_0(v)=\mathcal{F}^{-1}(\varphi_0)(v)$ is not a single Dirac mass, then   $f(t,\cdot)\in L^1(\mathbb{R}^3)\cap\mathcal{P}^\beta(\mathbb{R}^3)\cap H^\infty(\mathbb{R}^3)$ for $t>0$ and $0<\beta<\alpha$.
\end{Lemma}

And for self-similar solution $f_{\alpha,K}(t,v)$ constructed in
 \cite{Bobylev-Cercignani-JSP-2002a, Bobylev-Cercignani-JSP-2002b}, by
\cite{Morimoto-Wang-Yang-2014, Morimoto-Yang-2012}, we have

\begin{Lemma}\label{global-solvability-similarity} For $\alpha\in (2s,2)$ and
a constant $K>0$, there exists a radially symmetric function $\hat{\Psi}_{\alpha,K}(\xi)\in \mathcal{K}^\alpha(\mathbb{R}^3)$ satisfying \eqref{property-SS} such that
$$
f_{\alpha,K}(t,v)=e^{-3\mu_\alpha t}\Psi_{\alpha, K}\left(ve^{-\mu_\alpha t}\right)
$$
is a solution of the Cauchy problem \eqref{Boltzmann-equation} with initial datum $\Psi_{\alpha,K}(v)$. Moreover, $\Psi_{\alpha,K}(t,\cdot)\in L^1(\mathbb{R}^3)\cap\mathcal{P}^\beta(\mathbb{R}^3)\cap H^\infty(\mathbb{R}^3)$ for
$0<\beta<\alpha$.
\end{Lemma}

The relation between $\mathcal{P}^\alpha(\mathbb{R}^3)$ and $\mathcal{K}^\alpha(\mathbb{R}^3)$ was given in \cite{Cannone-Karch-CPAM-2010} and \cite{Morimoto-Wang-Yang-2014} and it can be stated as follows.

\begin{Lemma}\label{imbedding} It holds that
\begin{itemize}
\item [(i).] For $\alpha\in(0,2]$, if $h(v)\in \mathcal{P}^\alpha(\mathbb{R}^3)$, then $\hat{h}(\xi)\in\mathcal{K}^\alpha(\mathbb{R}^3)$;
\item[(ii).] For $\alpha\in(0,2]$, if $\hat{h}(\xi)\in\mathcal{K}^\alpha(\mathbb{R}^3)$, then for any $0<\beta<\alpha$, $h(v)\in \mathcal{P}^\beta(\mathbb{R}^3)$.
\end{itemize}
\end{Lemma}

Since the energy of the initial data is infinite, for analysis, we will
first approximate it by a cutoff on the large velocity so that the moment of
any order is bounded. And then it remains to show that the solution with
this kind of approximation has uniform bound  independent of the
cutoff paremeter. On the other hand, the approximate solution can not
be arbitrary because it has to be in the function space $\mathcal{K}^\alpha$.

 For $\alpha\in(2s,2]$ and $f_0(v)\in\widetilde{\mathcal{P}}^\alpha(\mathbb{R}^3)$, let $X(v)$ be the smooth function defined in the construction of $\widetilde{{P}}(t,\xi)$ and set $X_R(v)=X(v/R)$, define
\begin{equation}\label{approximation-f-0}
f_{0R}(v)=\widetilde{f}_{0R}\left(v+a^f_R\right),\quad \widetilde{f}_{0R}(v)=\frac{f_0(v)X_R(v)}{\int\limits_{\mathbb{R}^3}
f_0(v)X_R(v)dv}
\end{equation}
with
\begin{equation}\label{a-f-R}
a^f_R=\int\limits_{\mathbb{R}^3}v\widetilde{f}_{0R}(v)dv=\frac{\int\limits_{\mathbb{R}^3}vf_0(v)X_R(v)dv}{\int\limits_{\mathbb{R}^3}f_0(v)X_R(v)dv}.
\end{equation}

The properties of the approximation function are given in
\begin{Lemma}\label{properties-of-f-0R} For  $1<\beta<\alpha\leq 2$, if we choose $R>0$ sufficiently large, then
\begin{itemize}
\item[(i).]  $\hat{f}_{0R}(\xi), \hat{g}_{0R}(\xi)\in {\mathcal{K}}^2(\mathbb{R}^3)$, and  for sufficiently large $R>0$ it holds
\begin{equation}\label{estimate-IA-1}
\left\|\hat{f}_{0R}(\cdot)-\hat{g}_{0R}(\cdot)\right\|_{\mathcal{D}^2}\leq C_7\left( 1+\int\limits_{\mathbb{R}^3}|v|\left(f_0(v)+g_0(v)\right)dv+
\int\limits_{\mathbb{R}^3}|v|^2\left|f_0(v)-g_0(v)\right|dv\right).
\end{equation}
Here the positive constant $C_7$ depends only on $\int\limits_{\mathbb{R}^3}(1+|v|^\beta)(f_0(v)+g_0(v))dv$;
\item[(ii).] For $1<\beta<\alpha\leq 2$ and sufficiently large $R>0$, $f_{0R}(v)\in \mathcal{P}^\beta(\mathbb{R}^3)$ with $\mathcal{P}^\beta(\mathbb{R}^3)-$norm being uniformly bounded, precisely,
\begin{equation}\label{estimate-IA-2}
\int\limits_{\mathbb{R}^3}|v|^\beta f_{0R}(v)dv\lesssim \int\limits_{\mathbb{R}^3}\left(1+|v|^\beta\right) f_{0}(v)dv.
\end{equation}
Thus
\begin{equation}\label{estimate-IA-3}
\left\|\hat{f}_{0R}(\cdot)-1\right\|_{\mathcal{D}^\beta}\lesssim 1,\quad \left\|\hat{f}_{0R}(\cdot)-\hat{f}_0(\cdot)\right\|_{\mathcal{D}^\beta}\lesssim 1,
\end{equation}
and
\begin{equation}\label{estimate-IA-4}
\lim\limits_{R\to+\infty}\left\|\hat{f}_{0R}(\cdot)-\hat{f}_0(\cdot)\right\|
_{\mathcal{D}^\beta}=0.
\end{equation}
\end{itemize}
\end{Lemma}
\begin{proof} We first prove \eqref{estimate-IA-2}-\eqref{estimate-IA-4}. Since it is straightforward to verify
\eqref{estimate-IA-2}, {{}  \eqref{estimate-IA-3}} is a direct consequence of \eqref{estimate-IA-2} and Lemma \ref{imbedding}.
We only  prove \eqref{estimate-IA-4} as follows: For this, note that
$$
\lim\limits_{R\to+\infty}\int\limits_{\mathbb{R}^3}f_0(v)X_R(v)dv=1.
$$
Choose $R$ sufficiently large, we have
\begin{equation}\label{2.10}
\int\limits_{\mathbb{R}^3}f_0(v)X_R(v)dv\geq \frac 12,\quad \int\limits_{\mathbb{R}^3}g_0(v)X_R(v)dv\geq \frac 12.
\end{equation}
Thus
\begin{eqnarray}\label{2.11}
\left|a^f_R\right|&\leq& 2\left|\int\limits_{\mathbb{R}^3}vf_0(v)X_R(v)dv\right|\\
&=&2\left|\int\limits_{\mathbb{R}^3}vf_0(v)\left(1-X_R(v)\right)dv\right|
\leq 2R^{1-\beta}\int\limits_{\mathbb{R}^3}|v|^\beta f_0(v)dv.\nonumber
\end{eqnarray}
Similarly,
\begin{equation}\label{2.12}
\left|a^g_R\right|\leq 2R^{1-\beta}\int\limits_{\mathbb{R}^3}|v|^\beta g_0(v)dv.
\end{equation}
From \eqref{2.10}, \eqref{2.11}, and the fact that
\begin{eqnarray*}
\left|\hat{f}_{0R}(\xi)-\hat{f}_0(\xi)\right|&\leq&\int\limits_{\mathbb{R}^3}\left|f_{0R}(v)-f_0(v)\right|dv\\
&\leq& \left(\int\limits_{\mathbb{R}^3}f_0(v)X_R(v)dv\right)^{-1}\int\limits_{\mathbb{R}^3}\left|f_0\left(v+a^f_R\right)
-f_0(v)\right|dv\\
&&+\left(\int\limits_{\mathbb{R}^3}f_0(v)X_R(v)dv\right)^{-1}
\int\limits_{\mathbb{R}^3}f_0(v)\left|X_R\left(v+a^f_R\right)-\int\limits_{\mathbb{R}^3}f_0(v)X_R(v)dv\right|dv,
\end{eqnarray*}
we obtain

\begin{equation}\label{2.13}
\lim\limits_{R\to+\infty}\sup\limits_{\xi\in{\mathbb{R}^3}}\left|\hat{f}_{0R}(\xi)-\hat{f}_0(\xi)\right|=0.
\end{equation}
On the other hand,  $\hat{f}_0(\xi)\in \mathcal{K}^\alpha(\mathbb{R}^3)$
implies that $\left\|1-\hat{f}_0\right\|_{\mathcal{D}^\alpha}\lesssim 1$.
Consequently, for $1<\beta<\alpha\leq 2$, it holds that
\begin{equation*}
\frac{\left|1-\hat{f}_0(\eta)\right|}{|\eta|^\beta}\leq \left\|1-\hat{f}_0\right\|_{\mathcal{D}^\alpha}|\eta|^{\alpha-\beta},
\end{equation*}
so that for each $\varepsilon>0$, there exists a $\delta_1(\varepsilon)>0$ such that
\begin{equation}\label{2.14}
\frac{\left|1-\hat{f}_0(\eta)\right|}{|\eta|^\beta}<\frac \varepsilon 2
\end{equation}
holds for any $|\eta|<\delta_1$.

Choose $\widetilde{\beta}\in(\beta,\alpha)$ so that
$\left\|1-\hat{f}_{0R}\right\|_{\mathcal{D}^{\widetilde{\beta}}}$ is bounded by a constant independent of $R$
because of \eqref{estimate-IA-3}. Then

\begin{equation}\label{2.15}
\frac{\left|1-\hat{f}_{0R}(\eta)\right|}{|\eta|^\beta}\leq \left\|1-\hat{f}_{0R}\right\|_{\mathcal{D}^{\widetilde{\beta}}}|\eta|^{\widetilde{\beta}-\beta}\lesssim|\eta|^{\widetilde{\beta}-\beta}
<\frac\varepsilon 2
\end{equation}
provided that $|\eta|<\delta_2(\varepsilon)$ for some sufficiently small $\delta_2>0$.

\eqref{2.14} together with \eqref{2.15} imply that for any $\varepsilon>0$ and $|\eta|<\delta=
\min\left\{\delta_1,\delta_2\right\}$, we have
\begin{equation}\label{2.16}
\frac{\left|\hat{f}_{0R}(\eta)-\hat{f}_0(\eta)\right|}{|\eta|^\beta}\leq \frac{\left|1-\hat{f}_{0R}(\eta)\right|}{|\eta|^{\beta}}+ \frac{\left|1-\hat{f}_{0}(\eta)\right|}{|\eta|^{\beta}}
<\varepsilon.
\end{equation}
And \eqref{estimate-IA-4} follows directly from \eqref{2.13} and \eqref{2.16}.

Now it remains to prove \eqref{estimate-IA-1}.
Set
\begin{equation}\label{2.17}
k(v,\xi)=\frac{e^{-iv\cdot\xi}+iv\cdot\xi-1}{|\xi|^2},
\end{equation}
then
\begin{eqnarray}\label{2.18}
\frac{\left|\hat{f}_{0R}(\xi)-\hat{g}_{0R}(\xi)\right|}{|\xi|^2}&=&
\left|\int\limits_{\mathbb{R}^3}\left(k\left(v-a^f_R,\xi\right)\frac{f_0(v)X_R(v)}{\int\limits_{\mathbb{R}^3}f_0(v)X_R(v)dv}
-k\left(v-a^g_R,\xi\right)\frac{g_0(v)X_R(v)}{\int\limits_{\mathbb{R}^3}g_0(v)X_R(v)dv}\right)dv\right|\nonumber\\
&\lesssim&\underbrace{\left|\int\limits_{\mathbb{R}^3}\left(k\left(v-a^f_R,\xi\right)-k\left(v-a^g_R,\xi\right)\right)
\frac{f_0(v)X_R(v)}{\int\limits_{\mathbb{R}^3}f_0(v)X_R(v)dv}dv\right|}_{I_1}\\
&&+\underbrace{\left|\int\limits_{\mathbb{R}^3}k\left(v-a^g_R,\xi\right)\left(\frac{f_0(v)X_R(v)}{\int\limits_{\mathbb{R}^3}f_0(v)X_R(v)dv}
-\frac{g_0(v)X_R(v)}{\int\limits_{\mathbb{R}^3}g_0(v)X_R(v)dv}\right)dv\right|}_{I_2}.\nonumber
\end{eqnarray}
Firstly,  from \eqref{2.10}, \eqref{2.12} and the fact
$$
\left|k\left(v-a^g_R,\xi\right)\right|\lesssim \left|v-a^g_R\right|^2,
$$
we have

\begin{eqnarray}\label{2.19}
I_2&\leq& 2\left|\int\limits_{\mathbb{R}^3}k\left(v-a^g_R,\xi\right)(f_0(v)-g_0(v))X_R(v)dv\right|\nonumber\\
&&+4\left|\int\limits_{\mathbb{R}^3}k\left(v-a^g_R,\xi\right)g_0(v)X_R(v)
\left(\int\limits_{\mathbb{R}^3}(f_0(v)-g_0(v))X_R(v)dv\right)dv\right|\nonumber\\
&\lesssim& \int\limits_{\mathbb{R}^3}\left(1+|v|^2\right)\left|f_0(v)-g_0(v)\right|dv
+\int\limits_{\mathbb{R}^3}\left(1+|v|^2\right)g_0(v)X_R(v)
\left|\int\limits_{\mathbb{R}^3}(f_0(v)-g_0(v))\left(1-X_R(v)\right)dv\right|dv\\
&\lesssim&\int\limits_{\mathbb{R}^3}\left(1+|v|^2\right)\left|f_0(v)-g_0(v)\right|dv
+\int\limits_{\mathbb{R}^3}\left(1+|v|^\beta\right)g_0(v)dv\cdot R^{2-\beta}\cdot
\left|\int\limits_{\mathbb{R}^3}(f_0(v)-g_0(v))\left(1-X_R(v)\right)dv\right|\nonumber\\
&\lesssim& \int\limits_{\mathbb{R}^3}\left(1+|v|^2\right)\left|f_0(v)-g_0(v)\right|dv
+\int\limits_{\mathbb{R}^3}\left(1+|v|^\beta\right)g_0(v)dv\cdot \int\limits_{\mathbb{R}^3}|v|^{2-\beta}\left|f_0(v)-g_0(v)\right|dv\nonumber\\
&\lesssim&\int\limits_{\mathbb{R}^3}\left(1+|v|^2\right)\left|f_0(v)-g_0(v)\right|dv\cdot
\left(1+\int\limits_{\mathbb{R}^3}\left(1+|v|^\beta\right)g_0(v)dv\right).\nonumber
\end{eqnarray}
For $I_1$, by noticing
\begin{equation}\label{2.20}
\left|k\left(v-a^f_R,\xi\right)-k\left(v-a^g_R,\xi\right)\right|=
|\xi|^{-2}\left|e^{-i\left(v-a^f_R\right)\cdot\xi}
\left(e^{-i\left(a^g_R-a^f_R\right)\cdot\xi}-1\right)+i\left(a^g_R-a^f_R\right)\cdot\xi\right|,
\end{equation}
we have for $|\xi|\geq 1$ that
\begin{equation}\label{2.21}
\left|k\left(v-a^f_R,\xi\right)-k\left(v-a^g_R,\xi\right)\right|\lesssim \left|a^g_R-a^f_R\right|\lesssim R^{1-\beta}.
\end{equation}
For $|\xi|\leq 1$, it holds that
\begin{eqnarray}\label{2.22}
&&\left|k\left(v-a^f_R,\xi\right)-k\left(v-a^g_R,\xi\right)\right|\nonumber\\
&\lesssim& |\xi|^{-2}\left|\left[
\left(1+O(1)\left|v-a^g_R\right||\xi|\right)\left(-i\left(a^g_R-a^f_R\right)\cdot\xi+O(1)\left|a^g_R-a^f_R\right|^2|\xi|^2\right)
+i\left(a^g_R-a^f_R\right)\cdot\xi\right]\right|\nonumber\\
&\lesssim&\left|a^g_R-a^f_R\right|^2+\left|v-a^g_R\right|\left|a^g_R-a^f_R\right|
+\left|v-a^g_R\right|\left|a^g_R-a^f_R\right|^2|\xi|\\
&\lesssim& (1+|v|)R^{1-\beta}.\nonumber
\end{eqnarray}
Thus,  \eqref{2.20}, \eqref{2.21} and \eqref{2.22} imply that
\begin{equation}\label{2.23}
\left|k\left(v-a^f_R,\xi\right)-k\left(v-a^g_R,\xi\right)\right|\lesssim 1+|v|,
\end{equation}
and consequently
\begin{equation}\label{2.24}
I_1\lesssim \int\limits_{\mathbb{R}^3}(1+|v|)f_0(v)dv.
\end{equation}
Inserting \eqref{2.19} and \eqref{2.24} into \eqref{2.18} yields \eqref{estimate-IA-1}  and this completes the proof of Lemma \ref{properties-of-f-0R}.
\end{proof}

Now let
\begin{equation}\label{2.25}
P^R_{jl}(0)\equiv \int\limits_{\mathbb{R}^3}
\left(v_jv_l-\frac{\delta_{jl}}{3}|v|^2\right)\left(f_{0R}(v)-g_{0R}(v)\right)dv
\end{equation}
be the approximation of $P_{jl}(0)$ defined by \eqref{correction-funtion}$_3$.
The following lemma gives  the convergence of $P^R_{jl}(0)$ to $P_{jl}(0)$ as $R\to+\infty$.

\begin{Lemma}\label{convergence-correction-function}
Assume
\begin{equation}\label{2.26}
\int\limits_{\mathbb{R}^3}|v|^2\left|f_0(v)-g_0(v)\right|dv<+\infty,
\end{equation}
then
\begin{equation}\label{2.27}
\lim\limits_{R\to+\infty}P_{jl}^R(0)=P_{jl}(0).
\end{equation}
\end{Lemma}
\begin{proof}
Notice that

\begin{eqnarray}\label{2.28}
P^R_{jl}(0)&=&\int\limits_{\mathbb{R}^3}
\left(v_jv_l-\frac{\delta_{jl}}{3}|v|^2\right)\frac{f_0
\left(v+a^f_R\right)X_R\left(v+a^f_R\right)-f_0\left(v+a^g_R\right)X_R\left(v+a^g_R\right)}
{\int\limits_{\mathbb{R}^3}f_0(v)X_R(v)dv}dv\nonumber\\
&&+\left(\int\limits_{\mathbb{R}^3}f_0(v)X_R(v)dv\right)^{-1}\int\limits_{\mathbb{R}^3}\left(v_jv_l-\frac{\delta_{jl}}{3}|v|^2\right)
\left(f_0\left(v+a^g_R\right)-g_0\left(v+a^g_R\right)\right)X_R\left(v+a^g_R\right)dv\\
&&+\int\limits_{\mathbb{R}^3}
\left(v_jv_l-\frac{\delta_{jl}}{3}|v|^2\right)g_0\left(v+a^g_R\right)X_R\left(v+a^g_R\right)\left(
\left(\int\limits_{\mathbb{R}^3}f_0(v)X_R(v)dv\right)^{-1}-\left(\int\limits_{\mathbb{R}^3}g_0(v)X_R(v)dv\right)^{-1}\right)dv\nonumber\\
&=& \underbrace{\int\limits_{\mathbb{R}^3}\left[
\left(\left(v_j-a^f_{Rj}\right)\left(v_l-a^f_{Rl}\right)-\frac{\delta_{jl}}{3}\left|v-a^f_R\right|^2\right)
- \left(\left(v_j-a^g_{Rj}\right)\left(v_l-a^g_{Rl}\right)-\frac{\delta_{jl}}{3}\left|v-a^g_R\right|^2\right)\right]
\frac{f_0(v)X_R(v)}{\int\limits_{\mathbb{R}^3}f_0(v)X_R(v)dv}dv}_{I_3}\nonumber\\
&&+ \underbrace{\int\limits_{\mathbb{R}^3}\left(\left(v_j-a^g_{Rj}\right)\left(v_l-a^g_{Rl}\right)-\frac{\delta_{jl}}{3}\left|v-a^g_R\right|^2\right)
\frac{(f_0(v)-g_0(v))X_R(v)}{\int\limits_{\mathbb{R}^3}f_0(v)X_R(v)dv}dv}_{I_4}\nonumber\\
&&+\underbrace{\int\limits_{\mathbb{R}^3}\left(v_jv_l-\frac{\delta_{jl}}{3}|v|^2\right)g_0\left(v+a^g_R\right)X_R\left(v+a^g_R\right)
\frac{\int\limits_{\mathbb{R}^3}(f_0(v)-g_0(v))X_R(v)dv}{\int\limits_{\mathbb{R}^3}f_0(v)X_R(v)dv\cdot \int\limits_{\mathbb{R}^3}g_0(v)X_R(v)dv}dv}_{I_5}.\nonumber
\end{eqnarray}
We have from
\begin{eqnarray*}
&&\left|\left(\left(v_j-a^f_{Rj}\right)\left(v_l-a^f_{Rl}\right)-\frac{\delta_{jl}}{3}\left|v-a^f_R\right|^2\right)
- \left(\left(v_j-a^g_{Rj}\right)\left(v_l-a^g_{Rl}\right)-\frac{\delta_{jl}}{3}\left|v-a^g_R\right|^2\right)\right|\\
&\lesssim&\left|a^f_R-a^g_R\right||v|+\left|a^g_R\right|^2+\left|a^f_R\right|^2,
\end{eqnarray*}
and \eqref{2.10} that for $R\geq R_1$
$$
|I_3|\lesssim\left(\left|a^g_R\right|+\left|a^f_R\right|\right)\int\limits_{\mathbb{R}^3}(1+|v|)f_0(v)dv.
$$
This together with \eqref{2.11}-\eqref{2.12} imply that
\begin{equation}\label{2.29}
\lim\limits_{R\to+\infty} I_3=0.
\end{equation}
For $I_5$, if $R$ is sufficiently large, we have from \eqref{2.10}, \eqref{2.12} and the assumption %\eqref{zero-2-perturbation}
\eqref{2.26} that

\begin{eqnarray*}
|I_5|&\lesssim& \left|\int\limits_{\mathbb{R}^3}\left(v_jv_l-\frac{\delta_{jl}}{3}|v|^2\right)g_0\left(v+a^g_R\right)X_R\left(v+a^g_R\right)
\left(\int\limits_{\mathbb{R}^3}(f_0(v)-g_0(v))X_R(v)dv\right)dv\right|\\
&\lesssim& \int\limits_{\mathbb{R}^3}\left(1+\left|v-a^g_R\right|^2\right)g_0\left(v+a^g_R\right)X_R\left(v+a^g_R\right)
\left|\int\limits_{\mathbb{R}^3}(f_0(v)-g_0(v))\left(1-X_R(v)\right)dv\right|dv\\
&\lesssim& \int\limits_{\mathbb{R}^3}(1+|v|^\beta)R^{2-\beta}g_0(v)X_R(v)
\left|\int\limits_{\mathbb{R}^3}(f_0(v)-g_0(v))\left(1-X_R(v)\right)dv\right|dv\\
&\lesssim&R^{2(1-\beta)}\left(\int\limits_{\mathbb{R}^3}(1+|v|^\beta)g_0(v)dv\right)
\left(\int\limits_{\mathbb{R}^3}|v|^\beta|f_0(v)-g_0(v)|dv\right)
\lesssim R^{2(1-\beta)}.
\end{eqnarray*}
Thus

\begin{equation}\label{2.30}
\lim\limits_{R\to+\infty} I_5=0.
\end{equation}
Finally for $I_4$, from \eqref{2.11}, \eqref{2.12}, the assumption \eqref{2.26}
and $\lim\limits_{R\to+\infty}\int\limits_{\mathbb{R}^3}f_0(v)X_R(v)dv=1$, the dominated convergence theorem yields
\begin{equation}\label{2.31}
\lim\limits_{R\to+\infty}I_4=\int\limits_{\mathbb{R}^3}\left(v_jv_l-\frac{\delta_{jl}}{3}|v|^2\right)(f_0(v)-g_0(v))dv=P_{jl}(0).
\end{equation}
Inserting \eqref{2.29}, \eqref{2.30} and \eqref{2.31} into \eqref{2.29}
gives \eqref{2.27}. This completes the proof of the lemma.
\end{proof}

In the last lemma of this section, a sufficient condition
for  \eqref{high-integrability} on $f_0(v)-g_0(v)$
used  in Theorem \ref{Thm1.1} is given.

\begin{Lemma}\label{K-2+delta-bound} Let $0<\delta\leq 1$, it holds that
\begin{equation}\label{2.32}
\left\|\hat{f}_0(\cdot)-\hat{g}_0(\cdot)-\widetilde{P}(0,\cdot)\right\|
_{\mathcal{D}^{2+\delta}}\lesssim
\int\limits_{\mathbb{R}^3}\left(1+|v|^{2+\delta}\right)|f_0(v)-g_0(v)|dv.
\end{equation}
\end{Lemma}
\begin{proof} In fact, by the assumption \eqref{zero-2-perturbation},
we have
\begin{eqnarray*}
\hat{f}_0(\xi)-\hat{g}_0(\xi)-\widetilde{P}(0,\xi)&=&\int\limits_{\mathbb{R}^3}
\left[e^{-iv\cdot\xi}-1+iv\cdot\xi-\frac 12\sum\limits_{j,l=1}^3\xi_j\xi_l X(\xi)\left(v_jv_l-\frac{\delta_{jl}}{3}|v|^2\right)\right](f_0(v)-g_0(v))dv\\
&=&\int\limits_{\mathbb{R}^3}
\left[e^{-iv\cdot\xi}-1+iv\cdot\xi-\frac 12\sum\limits_{j,l=1}^3\xi_j\xi_l X(\xi)v_jv_l\right](f_0(v)-g_0(v))dv.
\end{eqnarray*}
The Taylor expansion of $e^{-iv\cdot\xi}-1+iv\cdot\xi$ to the second order implies
that
$$
\left|e^{-iv\cdot\xi}-1+iv\cdot\xi-\frac 12\sum\limits_{j,l=1}^3\xi_j\xi_l X(\xi)v_jv_l\right|\lesssim |v|^2|\xi|^2,
$$
and the Taylor expansion to $e^{-iv\cdot\xi}-1+iv\cdot\xi-\frac 12\sum\limits_{j,l=1}^3\xi_j\xi_l X(\xi)v_jv_l$  to the third order gives
$$
\left|e^{-iv\cdot\xi}-1+iv\cdot\xi-\frac 12\sum\limits_{j,l=1}^3\xi_j\xi_l X(\xi)v_jv_l\right|\lesssim \left(1+|v|^3\right)|\xi|^3.
$$
Thus interpolation yields

$$
\left|e^{-iv\cdot\xi}-1+iv\cdot\xi-\frac 12\sum\limits_{j,l=1}^3\xi_j\xi_l X(\xi)v_jv_l\right|\lesssim \left(1+|v|^{2+\delta}\right)|\xi|^{2+\delta},
$$
for $0<\delta\leq 1$.
With this,  \eqref{2.32} follows. And
this completes the proof of the lemma.
\end{proof}

\section{Proof of Theorem \ref{Thm1.1}}
\setcounter{equation}{0}
To prove Theorem \ref{Thm1.1}, as in \cite{Cannone-Karch-CPAM-2010},
we first approximate the cross section by a sequence of bounded cross sections
defined by
\begin{equation}\label{3.1}
\mathcal{B}_n(s)=\min\left\{\mathcal{B}(s), n\right\},\quad n\in\mathbb{N}.
\end{equation}
Then consider
\begin{eqnarray}
&&\partial_tH_n+H_n=\frac{1}{\overline{\sigma}_n}\int\limits_{\mathbb{R}^3}\int\limits_{\mathbb{S}^2}
\mathcal{B}_n\left(\frac{(v-v_*)\cdot \sigma}{|v-v_*|}\right)H_n(v')H_n(v'_*)d\sigma dv_*,\label{3.3}\\
&&H_n(0,v)=H_0(v).\label{3.4}
\end{eqnarray}
Here
\begin{equation}\label{3.2}
\overline{\sigma}_n=\int\limits_{\mathbb{S}^2}\mathcal{B}_n\left(\frac{\xi\cdot \sigma}{|\xi|}\right)d\sigma.
\end{equation}

For $\alpha\in(\max\{2s,1\},2]$ and  $f_0(v), g_0(v)\in\widetilde{\mathcal{P}}^\alpha(\mathbb{R}^3)$, let $f_{0R}(v)$ and $g_{0R}(v)$ be the approximation of $f_0(v)$ and $g_0(v)$ constructed in the previous section. Since $f_{0R}(v), g_{0R}(v)\in \mathcal{P}^2(\mathbb{R}^3)\subset \mathcal{K}^2(\mathbb{R}^3)$, it follows from Lemma \ref{global-solvability} that the Cauchy problem \eqref{3.3}-\eqref{3.4} with $H_0(v)=f_{0R}(v)$ ( $H_0(v)=g_{0R}(v)$) admits a unique non-negative global solution $F^n_R(t,v)$ ($G^n_R(t,v)$) satisfying $\hat{F}^n_R(t,\xi)\in C\left([0,\infty),\mathcal{K}^2(\mathbb{R}^3)\right)$ ($\hat{G}^n_R(t,\xi)\in C\left([0,\infty),\mathcal{K}^2(\mathbb{R}^3)\right)$). Moreover, for  $\max\{2s,1\}<\beta<\alpha\leq 2$, \eqref{K-alpha-stability}, Lemma \ref{imbedding} and Lemma \ref{properties-of-f-0R} imply that
\begin{equation}\label{3.5}
\left\{
\begin{array}{l}
\left\|\hat{F}^n_R(t,\cdot)-\hat{F}_n(t,\cdot)\right\|_{\mathcal{D}^\beta}\leq e^{\lambda^n_\beta t}
\left\|\hat{f}_{0R}(\cdot)-\hat{f}_0(\cdot)\right\|_{\mathcal{D}^\beta}\lesssim e^{\lambda^n_\beta t},\\[3mm]
\left\|\hat{G}^n_R(t,\cdot)-\hat{G}_n(t,\cdot)\right\|_{\mathcal{D}^\beta}\leq e^{\lambda^n_\beta t}
\left\|\hat{g}_{0R}(\cdot)-\hat{g}_0(\cdot)\right\|_{\mathcal{D}^\beta}\lesssim e^{\lambda^n_\beta t}.
\end{array}
\right.
\end{equation}
Here $F_n(t,v)$ and $G_n(t,v)$ denote the unique non-negative solutions of the Cauchy problem \eqref{3.3}-\eqref{3.4} with initial data $f_0(v)\in \widetilde{\mathcal{P}}^\alpha(\mathbb{R}^3)$ and $g_0(v)\in \widetilde{\mathcal{P}}^\alpha(\mathbb{R}^3)$ respectively, and
\begin{equation}\label{3.6}
\lambda^n_\alpha=\frac{1}{\overline{\sigma}_n}\int\limits_{\mathbb{S}^2}\mathcal{B}_n\left(\frac{\xi\cdot\sigma}{|\xi|}\right)
\left(\frac{|\xi^+|^\alpha+|\xi^-|^\alpha}{|\xi|^\alpha}-1\right)d\sigma.
\end{equation}

Furthermore,
$$
\int\limits_{\mathbb{R}^3}|v|^2F^n_R(t,v)dv=\int\limits_{\mathbb{R}^3}|v|^2f_{0R}(v)dv<+\infty,\quad
\int\limits_{\mathbb{R}^3}|v|^2G^n_R(t,v)dv=\int\limits_{\mathbb{R}^3}|v|^2g_{0R}(v)dv<+\infty.
$$
Consequently, Lemma \ref{global-solvability} yields
\begin{equation}\label{3.7}
\left\|\hat{F}^n_R(t,\cdot)-\hat{G}^n_R(t,\cdot)\right\|_{\mathcal{D}^2}\leq
\left\|\hat{f}_{0R}(\cdot)-\hat{g}_{0R}(\cdot)\right\|_{\mathcal{D}^2}\lesssim 1,\quad t>0.
\end{equation}

Noticing that
$$
\left|\hat{F}^n_R(t,\xi)-\hat{F}_n(t,\xi)\right|\leq |\xi|^\beta\left\|\hat{F}^n_R(t,\cdot)-\hat{F}_n(t,\cdot)\right\|_{\mathcal{D}^\beta}\leq |\xi|^\beta e^{\lambda^n_\beta t}
\left\|\hat{f}_{0R}(\cdot)-\hat{f}_0(\cdot)\right\|_{\mathcal{D}^\beta},
$$
where \eqref{3.5} has been used,  from \eqref{estimate-IA-4}, we have

\begin{Lemma}\label{conv-f-n-R} The limit
$$
\lim\limits_{R\to+\infty}\left(\hat{F}^n_R(t,\xi),\hat{G}^n_R(t,\xi)\right)=\left(\hat{F}_n(t,\xi),\hat{G}_n(t,\xi)\right)
$$
holds uniformly, locally with respect to $t\in\mathbb{R}^+$ and $\xi\in\mathbb{R}^3$.
\end{Lemma}

{{} Putting}
\begin{equation}\label{3.8}
\Phi^{nR}_1(t,\xi)=\hat{F}^n_R(t,\xi)-\hat{G}^n_R(t,\xi)-\widetilde{P}^n_R(t,\xi)
\end{equation}
with
\begin{eqnarray}\label{3.9}
\widetilde{P}^n_R(t,\xi)&=&\frac 12 e^{-A_n t}\sum\limits_{j,l=1}^3P^R_{jl}(0)\xi_j\xi_lX(\xi),\\
A_n&=&\frac{3}{4\overline{\sigma}_n}\int\limits_{\mathbb{S}^2}\mathcal{B}_n\left(\frac{\xi\cdot\sigma}{|\xi|}\right)
\left[1-\left(\frac{\xi\cdot\sigma}{|\xi|}\right)^2\right]d\sigma,\nonumber
\end{eqnarray}
we now deduce the equation for $\Phi^{nR}_1(t,\xi)$. Set
\begin{equation}\label{3.10}
\hat{Q}^+_n\left(\hat{F},\hat{G}\right)=\frac{1}{\overline{\sigma}_n}\int\limits_{\mathbb{S}^2}\mathcal{B}_n\left(\frac{\xi\cdot\sigma}{|\xi|}\right)
\hat{F}(t,\xi^+)\hat{G}(t,\xi^-)d\sigma.
\end{equation}
Since $\hat{F}^n_R(t,\xi)$ and $\hat{G}^n_R(t,\xi)$ satisfy
$$
\left\{
\begin{array}{l}
\partial_t \hat{F}^n_R+\hat{F}^n_R=\hat{Q}^+\left(\hat{F}^n_R,\hat{F}^n_R\right),\\[3mm]
\partial_t \hat{G}^n_R+\hat{G}^n_R=\hat{Q}^+\left(\hat{G}^n_R,\hat{G}^n_R\right),
\end{array}
\right.
$$
we have
\begin{eqnarray}\label{3.11}
\partial_t\Phi^{nR}_1+\Phi^{nR}_1&=&-\left(\partial_t\widetilde{P}^n_R+\widetilde{P}^n_R\right)
+\hat{Q}^+\left(\Phi_1^{nR},\hat{F}^n_R\right)+\hat{Q}^+\left(\hat{G}^n_R,\Phi_1^{nR}\right)\nonumber\\
&&+\hat{Q}^+\left(\widetilde{P}^n_R,\hat{F}^n_R\right)
+\hat{Q}^+\left(\hat{G}^n_R,\widetilde{P}^n_R\right),\\
\Phi^{nR}_1(0,\xi)&=&\hat{f}_{0R}(\xi)-\hat{g}_{0R}(\xi)-\widetilde{P}^n_R(0,\xi).\nonumber
\end{eqnarray}

Let
\begin{equation}\label{3.12}
\Phi^{n}_1(t,\xi)=\hat{F}^n(t,\xi)-\hat{G}^n(t,\xi)-\widetilde{P}^n(t,\xi).
\end{equation}
By taking $R\to+\infty$, we have from Lemma \ref{conv-f-n-R} and Lemma \ref{convergence-correction-function} that $\Phi^{nR}_1(t,\xi)\to \Phi^n_1(t,\xi)$ uniformly, locally with respect to $t\in\mathbb{R}^+$ and $\xi\in\mathbb{R}^3$ as $R\to+\infty$. To derive the equation for $\Phi^n_1(t,\xi)$,
we firstly study
\begin{equation}\label{3.13}
E^n_R(t,\xi)=\hat{Q}^+\left(\widetilde{P}^n_R,\hat{F}^n_R\right)
+\hat{Q}^+\left(\hat{G}^n_R,\widetilde{P}^n_R\right)-\left(\partial_t\widetilde{P}^n_R(t,\xi)+\widetilde{P}^n_R(t,\xi)\right).
\end{equation}
In fact, for $E^n_R(t,\xi)$, we have
\begin{Lemma}\label{I-6-R} It holds that
\begin{eqnarray}\label{3.14}
\lim\limits_{R\to+\infty}E^n_R(t,\xi)=E^n(t,\xi),
\end{eqnarray}
uniformly, locally with respect to $t\in\mathbb{R}^+$ and $\xi\in\mathbb{R}^3$.
 And $E^n(t,\xi)$ satisfies
\begin{equation}\label{3.15}
\left|E^n(t,\xi)\right|\leq
\left\{
\begin{array}{rl}
O(1) |\xi|^{2+\delta}e^{-\left(A_n-\frac{\delta \lambda_\alpha^n}{\alpha} \right)t},&\quad |\xi|\leq 1,\\[2mm]
O(1)e^{-A_nt},&\quad |\xi|\geq 1
\end{array}
\right.
\end{equation}
for any $\delta\in (0,\alpha]\cap\left(0,\frac{\alpha A_n}{\lambda^n_\alpha}\right)$ and some positive constant $O(1)$ independent of $t, \xi, R$
and $n$.
\end{Lemma}
\begin{proof} Since
\begin{eqnarray*}
E^n_R(t,\xi)&=&-\left(\partial_t\widetilde{P}^n_R(t,\xi)+\widetilde{P}^n_R(t,\xi)\right)\\
&&+\frac 1{2\overline{\sigma}_n}\sum\limits_{j,l=1}^3e^{-A_nt}P^R_{jl}(0)
\int\limits_{\mathbb{S}^2}\mathcal{B}_n\left(\frac{\xi\cdot\sigma}{|\xi|}\right)\Big[\xi_j^+\xi_l^+X(\xi^+)
+\xi_j^-\xi_l^-X(\xi^-)\Big]d\sigma\\
&&+\frac 1{2\overline{\sigma}_n}\sum\limits_{j,l=1}^3e^{-A_nt}P^R_{jl}(0)
\int\limits_{\mathbb{S}^2}\mathcal{B}_n\left(\frac{\xi\cdot\sigma}{|\xi|}\right)\left[\xi_j^+\xi_l^+X(\xi^+)\left(\hat{F}^n_R(t,\xi^-)-1\right)
+\xi_j^-\xi_l^-X(\xi^-)\left(\hat{G}^n_R(t,\xi^+)-1\right)\right]d\sigma,
\end{eqnarray*}
{{} it follows} from Lemma \ref{conv-f-n-R} and Lemma \ref{convergence-correction-function} that
\begin{equation}\label{3.16}
\lim\limits_{R\to+\infty}E^n_R(t,\xi)=E^n(t,\xi)
\end{equation}
uniformly, locally with respect to $t\in\mathbb{R}^+$ and $\xi\in\mathbb{R}^3$. Here
\begin{eqnarray}\label{3.17}
E^n(t,\xi)&=&-\left(\partial_t\widetilde{P}^n(t,\xi)+\widetilde{P}^n(t,\xi)\right)+\underbrace{\frac 1{2\overline{\sigma}_n}\sum\limits_{j,l=1}^3e^{-A_nt}P_{jl}(0)
\int\limits_{\mathbb{S}^2}\mathcal{B}_n\left(\frac{\xi\cdot\sigma}{|\xi|}\right)\Big[\xi_j^+\xi_l^+X(\xi^+)
+\xi_j^-\xi_l^-X(\xi^-)\Big]d\sigma}_{I_6}\\
&&+\underbrace{\frac 1{2\overline{\sigma}_n}\sum\limits_{j,l=1}^3e^{-A_nt}P_{jl}(0)
\int\limits_{\mathbb{S}^2}\mathcal{B}_n\left(\frac{\xi\cdot\sigma}{|\xi|}\right)\left[\xi_j^+\xi_l^+X(\xi^+)\left(\hat{F}^n(t,\xi^-)-1\right)
+\xi_j^-\xi_l^-X(\xi^-)\left(\hat{G}^n(t,\xi^+)-1\right)\right]d\sigma}_{I_7},\nonumber
\end{eqnarray}
and
\begin{equation}\label{3.18}
\widetilde{P}^n(t,\xi)=\frac 12 e^{-A_n t}\sum\limits_{j,l=1}^3P_{jl}(0)\xi_j\xi_lX(\xi).
\end{equation}
To estimate the  bounds on $I_6$ and $I_7$, firstly note that for $|\xi|\geq 1$,
$$
\left|\hat{F}^n(t,\xi)\right|\leq 1,\quad \left|\hat{G}^n(t,\xi)\right|\leq 1
$$
{{} imply} that
\begin{equation}\label{3.19}
\left|E^n(t,\xi)\right|\leq O(1)(1+A_n)e^{-A_nt}\leq O(1)e^{-A_nt}.
\end{equation}
Here we have used the fact that $A_n$ has a uniform upper bound for any $n\in\mathbb{N}$.

If $|\xi|\leq 1$, then $|\xi^\pm|\leq 1$ so that $X(\xi^\pm)\equiv 1$. Hence, as obtained in \cite{Gabetta-Toscani-Wennberg-JSP-1995}, we have
$$
\xi_j^-\xi_l^-+\xi_j^+\xi_l^+= \frac 12\left(\xi_j\xi_l+|\xi|^2\sigma_j\sigma_l\right),
$$
and
$$
\frac{1}{\overline{\sigma}_n}\int\limits_{\mathbb{S}^2}\mathcal{B}_n\left(\frac{\xi\cdot\sigma}{|\xi|}\right)\sigma_j\sigma_ld\sigma
=\frac{2A_n}{3}\delta_{jl}+\left(1-2A_n\right)\frac{\xi_j\xi_l}{|\xi|^2}.
$$
Then
\begin{eqnarray*}
I_6&=&\frac 1{2\overline{\sigma}_n}\sum\limits_{j,l=1}^3e^{-A_nt}P_{jl}(0)
\int\limits_{\mathbb{S}^2}\mathcal{B}_n\left(\frac{\xi\cdot\sigma}{|\xi|}\right)\Big[\xi_j^+\xi_l^+
+\xi_j^-\xi_l^-\Big]d\sigma\\
&=&\frac 1{4\overline{\sigma}_n}\sum\limits_{j,l=1}^3e^{-A_nt}P_{jl}(0)
\int\limits_{\mathbb{S}^2}\mathcal{B}_n\left(\frac{\xi\cdot\sigma}{|\xi|}\right)\Big[\xi_j\xi_l
+|\xi|^2\sigma_j\sigma_l\Big]d\sigma\\
&=&\frac 14\sum\limits_{j,l=1}^3e^{-A_nt}P_{jl}(0)\xi_j\xi_l
+\frac 14\sum\limits_{j,l=1}^3e^{-A_nt}P_{jl}(0)\left[(1-2A_n)\xi_j\xi_l+\frac{2A_n}{3}\delta_{jl}|\xi|^2\right]\\
&=&(1-A_n)\widetilde{P}^n(t,\xi)= \partial_t\widetilde{P}^n(t,\xi)+\widetilde{P}^n(t,\xi).
\end{eqnarray*}
Thus for $|\xi|\leq 1$, it holds that
\begin{equation}\label{3.20}
I_6-\left(\partial_t\widetilde{P}^n(t,\xi)+\widetilde{P}^n(t,\xi)\right)=0.
\end{equation}
For $I_7$ when $|\xi|\leq 1$, we have from the assumption $f_0(v), g_0(v)\in \widetilde{\mathcal{P}}^\alpha(\mathbb{R}^3)$ and Lemma \ref{global-solvability} that
\begin{eqnarray*}
&&\left|\hat{F}_n(t,\xi)-1\right|\leq |\xi|^\alpha\left\|\hat{F}_n(t,\cdot)-1\right\|_{\mathcal{D}^\alpha}\leq e^{\lambda_\alpha^nt}|\xi|^\alpha \left\|\hat{f}_0(\cdot)-1\right\|_{\mathcal{D}^\alpha},\\
&&\left|\hat{G}_n(t,\xi)-1\right|\leq |\xi|^\alpha\left\|\hat{G}_n(t,\cdot)-1\right\|_{\mathcal{D}^\alpha}\leq e^{\lambda_\alpha^nt}|\xi|^\alpha \left\|\hat{g}_0(\cdot)-1\right\|_{\mathcal{D}^\alpha}.
\end{eqnarray*}
The above estimates together with $\left|\hat{F}^n(t,\xi)\right|\leq 1
$ and $ \left|\hat{G}^n(t,\xi)\right|\leq 1$ imply that
\begin{equation}\label{3.21}
\left|\hat{F}_n(t,\xi)-1\right|+\left|\hat{G}_n(t,\xi)-1\right|\leq O(1)|\xi|^{\varepsilon\alpha}e^{\varepsilon \lambda_\alpha^nt}
\end{equation}
 for any $\varepsilon\in(0,1]$. Consequently, for $|\xi|\leq 1$,
\begin{equation}\label{3.22}
|I_7|\leq O(1) |\xi|^{2+\varepsilon\alpha}e^{-\left(A_n-\varepsilon\lambda_\alpha^n\right)t}.
\end{equation}
\eqref{3.20} together with \eqref{3.22} imply that
\begin{equation}\label{3.23}
\left|E^n(t,\xi)\right|\leq O(1) |\xi|^{2+\varepsilon\alpha}e^{-\left(A_n-\varepsilon\lambda_\alpha^n\right)t}, \quad |\xi|\leq 1.
\end{equation}

With \eqref{3.19} and \eqref{3.23}, let $\delta=\varepsilon\alpha$, the estimate \eqref{3.15} follows immediately. This completes the proof of the
lemma.
\end{proof}

Now by letting $R\to+\infty$ in \eqref{3.11}, we get from Lemma \ref{convergence-correction-function}, Lemma \ref{conv-f-n-R} and Lemma \ref{I-6-R} that $\Phi^{n}_1(t,\xi)=\hat{F}^n(t,\xi)-\hat{G}^n(t,\xi)-\widetilde{P}^n(t,\xi)$ solves
\begin{eqnarray}
\partial_t\Phi^{n}_1+\Phi^{n}_1&=&\hat{Q}^+\left(\Phi_1^{n},\hat{F}^n\right)+\hat{Q}^+\left(\hat{G}^n,\Phi_1^{n}\right)
+E^n(t,\xi),\label{3.24}\\
\Phi^{n}_1(0,\xi)&=&\hat{f}_{0}(\xi)-\hat{g}_{0}(\xi)-\widetilde{P}^n(0,\xi).\label{3.25}
\end{eqnarray}
Here $E^n(t,\xi)$ satisfies \eqref{3.15}. By Lemmas \ref{convergence-correction-function},  \ref{conv-f-n-R} and \ref{I-6-R},  $\Phi^n_1(t,\xi)$, $\hat{F}^n(t,\xi)$, $\hat{G}_n(t,\xi)$ and $E^n(t,\xi)$ are continuous functions of $(t,\xi)\in \mathbb{R}^+\times\mathbb{R}^3$ and satisfy \eqref{3.24} in the sense of distribution. Since $\Phi^n_1(t,\xi)$, $\hat{F}^n(t,\xi)$, $\hat{G}_n(t,\xi)$, and $E^n(t,\xi)$  are uniformly bounded,  $\partial_t\Phi^n_1(t,\xi)$ is also uniformly bounded so that $\Phi^n_1(t,\xi)$ is globally Lipschitz continuous with respect to $t$. Hence \eqref{3.24}  holds almost everywhere.
Furthermore, by the continuity of $\Phi^n_1(t,\xi)$, $\hat{F}^n(t,\xi)$, $\hat{G}_n(t,\xi)$ and $E^n(t,\xi)$, we have that $\partial_t\Phi^n_1(t,\xi)$ is a continuous function of $(t,\xi)\in \mathbb{R}^+\times\mathbb{R}^3$ and consequently $\Phi^n_1(t,\xi)$ satisfies \eqref{3.24} everywhere.

The next lemma is about the  upper bound on $\left\|\Phi^n_1(t,\cdot)\right\|_{\mathcal{D}^{2+\delta}}$  for some $\delta\in(0,\alpha]\cap\left(0,\frac{\alpha A_n}{\lambda_n}\right)$.
\begin{Lemma}\label{Lemma-3.3} If $\left\|\Phi^n_1(0,\cdot)\right\|_{\mathcal{D}^{2+\delta}}<+\infty$ with $\delta\in (0,\alpha]\cap\left(0,\frac{\alpha A_n}{\lambda_n}\right)$, then
\begin{equation}\label{3.26}
\left\|\Phi^n_1(t,\cdot)\right\|_{\mathcal{D}^{2+\delta}}\lesssim e^{-\eta^n_0t}.
\end{equation}
Here $\eta^n_0=\min\left\{B_n,A_n-\frac{\delta\lambda^n_\alpha}{\alpha}\right\}$ with
\begin{equation}\label{3.27}
B_n=\frac{1}{\overline{\sigma}_n}\int\limits_{\mathbb{S}^2}\mathcal{B}_n\left(\frac{\sigma\cdot\xi}{|\xi|}\right)\left(1-\left|\cos\frac\theta 2\right|^{2+\delta}-\left|\sin\frac\theta 2\right|^{2+\delta}\right)d\sigma,\quad \cos\theta=\frac{\sigma\cdot\xi}{|\xi|}.
\end{equation}
\end{Lemma}
\begin{proof}
The proof is divided into two steps, the first step is to show that $\frac{\Phi^n_1(t,\xi)}{|\xi|^{2+\delta}}\in L^\infty(\mathbb{R}^3)$. Indeed, for $\kappa>0$, we have from \eqref{3.24} that
\begin{eqnarray}\label{3.28}
&&\left(\frac{\Phi^n_1(t,\xi)}{|\xi|^{2}\left(|\xi|^\delta+\kappa\right)}\right)_t+\frac{\Phi^n_1(t,\xi)}
{|\xi|^{2}\left(|\xi|^\delta+\kappa\right)}\nonumber\\
&=&\frac{1}{\overline{\sigma}_n}\int\limits_{\mathbb{S}^2}\mathcal{B}_n\left(\frac{\sigma\cdot\xi}{|\xi|}\right)
\left[\frac{\Phi^n_1(t,\xi^+)}{|\xi^+|^{2}\left(|\xi^+|^\delta+\kappa\right)}\frac{|\xi^+|^{2}\left(|\xi^+|^\delta+\kappa\right)}
{|\xi|^{2}\left(|\xi|^\delta+\kappa\right)}\hat{F}^n(t,\xi^-)\right.\\
&&\left.+\hat{G}^n(t,\xi^+)
\frac{\Phi^n_1(t,\xi^-)}{|\xi^-|^{2}\left(|\xi^-|^\delta+\kappa\right)}\frac{|\xi^-|^{2}\left(|\xi^-|^\delta+\kappa\right)}
{|\xi|^{2}\left(|\xi|^\delta+\kappa\right)}\right]d\sigma +\frac{E^n(t,\xi)}{|\xi|^2\left(|\xi|^\delta+\kappa\right)}.\nonumber
\end{eqnarray}

On the other hand, by letting $R\to+\infty$ in \eqref{3.7}, we have from Lemma \ref{conv-f-n-R} that for $t>0$
\begin{equation}\label{3.29}
\left\|\hat{F}^n(t,\cdot)-\hat{G}^n(t,\cdot)\right\|_{\mathcal{D}^2}\leq
\left\|\hat{f}_{0}(\cdot)-\hat{g}_{0}(\cdot)\right\|_{\mathcal{D}^2}\lesssim 1.
\end{equation}

\eqref{3.29} together with the definition of $\widetilde{P}^n(t,\xi)$ imply that
$$
\frac{\Phi^n_1(t,\xi)}{|\xi|^{2}\left(|\xi|^\delta+\kappa\right)}\in L^\infty(\mathbb{R}^3),\quad t>0,
$$
for  any $\kappa>0$. Hence,  by \eqref{3.28}, Lemma \ref{I-6-R} and the fact
that
$|\xi^\pm|\leq |\xi|, |\hat{F}^n(t,\xi)|\leq 1,  |\hat{G}^n(t,\xi)|\leq 1$, we can deduce by using the Gronwall inequality that there exists a positive constant $C(T)>0$ independent of $\kappa, n$ and $\xi$ such that
\begin{equation}\label{3.30}
\sup\limits_{0\not=\xi\in\mathbb{R}^3}\left\{\frac{\left|\Phi^n_1(t,\xi)\right|}{|\xi|^{2}\left(|\xi|^\delta+\kappa\right)}\right\}
\leq C(T),
\end{equation}
holds for $0\leq t\leq T$. Here $T>0$ is any given positive constant.

Since the positive constant $C(T)>0$ in \eqref{3.30} is independent of $\kappa$, we have from \eqref{3.30} by letting $\kappa\to 0_+$ that
\begin{equation}\label{3.31}
\sup\limits_{0\not=\xi\in\mathbb{R}^3}\left\{\frac{\left|\Phi^n_1(t,\xi)\right|}{|\xi|^{2+\delta}}\right\}
\leq C(T), \quad 0\leq t\leq T.
\end{equation}

With \eqref{3.31}, set
\begin{equation}\label{3.32}
\Phi^n_2(t,\xi)=\frac{\Phi^n_1(t,\xi)}{|\xi|^{2+\delta}},
\end{equation}
we can get from \eqref{3.24} and the fact $|\xi^\pm|^2=\frac{|\xi|^2\pm\xi\cdot\sigma|\xi|}{2}=\frac{|\xi|^2(1\pm\cos\theta)}{2}$ that
\begin{eqnarray}\label{3.33}
\partial_t\Phi^n_2+\Phi^n_2&=&\frac{1}{\overline{\sigma}_n}\int\limits_{\mathbb{S}^2}\mathcal{B}_n
\left(\frac{\xi\cdot\sigma}{|\xi|}\right)\left[
\Phi^n_2(t,\xi^+)\hat{F}^n(t,\xi^-)\left|\cos\frac\theta 2\right|^{2+\delta}\right.\\
&&\left.+\hat{G}^n(t,\xi^+)\Phi^n_2(t,\xi^-)\left|\sin\frac\theta 2\right|^{2+\delta}\right]d\sigma+\frac{E^n(t,\xi)}{|\xi|^{2+\delta}}.\nonumber
\end{eqnarray}
A direct consequence of \eqref{3.33} yields
\begin{equation}\label{3.34}
\left|\partial_t\Phi^n_2+\Phi^n_2\right|\leq (1-B_n)\left\|\Phi^n_2(t,\cdot)\right\|_{L^\infty}
+O(1)e^{-(A_n-\delta\lambda^n_\alpha/\alpha)t}.
\end{equation}
Since $0<B_n<1$, we can apply the argument used in \cite{Gabetta-Toscani-Wennberg-JSP-1995} to have
\begin{equation}\label{3.35}
\left|\Phi^n_2(t,\xi)\right|\leq O(1)e^{-\eta_0t},
\end{equation}
so that \eqref{3.26} follows. This completes the proof of the
lemma.
\end{proof}

We now turn to prove Theorem \ref{Thm1.1}. Let $F_n(t,v)$ and $G_n(t,v)$ be the unique solutions of the Cauchy problem \eqref{3.3}-\eqref{3.4} with initial data $f_0(v)$ and $g_0(v)$ respectively, then
\begin{equation}\label{3.36}
\left\{
\begin{array}{l}
f_n(t,v)\equiv F_n\left(\overline{\sigma}_nt,v\right),\\[2mm]
g_n(t,v)\equiv G_n\left(\overline{\sigma}_nt,v\right),
\end{array}
\right.
\end{equation}
solve
\begin{equation}\label{3.37}
\left\{
\begin{array}{l}
\partial_tf_n+\overline{\sigma}_nf_n={\displaystyle\int\limits_{\mathbb{S}^2}}\mathcal{B}_n\left(\frac{\sigma\cdot(v-v_*)}{|v-v_*|}\right)
f_n(v')f_n(v'_*)dv_*d\sigma,\\[3mm]
f_n(0,v)=f_0(v),
\end{array}
\right.
\end{equation}
and
\begin{equation}\label{3.38}
\left\{
\begin{array}{l}
\partial_tg_n+\overline{\sigma}_ng_n={\displaystyle\int\limits_{\mathbb{S}^2}}\mathcal{B}_n\left(\frac{\sigma\cdot(v-v_*)}{|v-v_*|}\right)
g_n(v')g_n(v'_*)dv_*d\sigma,\\[3mm]
g_n(0,v)=g_0(v),
\end{array}
\right.
\end{equation}
respectively.

The estimate \eqref{3.26}  in Lemma \ref{Lemma-3.3} gives
\begin{equation}\label{3.39}
\left\|\hat{F}_n(t,\cdot)-\hat{G}_n(t,\cdot)-\widetilde{P}^n(t,\cdot)\right\|
_{\mathcal{D}^{2+\delta}}\leq O(1)e^{-\eta_0^nt}.
\end{equation}

Putting \eqref{3.36} and \eqref{3.39} together yields
\begin{equation}\label{3.40}
\left\|\hat{f}_n(t,\cdot)-\hat{g}_n(t,\cdot)
-\widetilde{P}^n\left(\overline{\sigma}_nt,\cdot\right)\right\|_{\mathcal{D}^{2+\delta}}\leq O(1)e^{-\overline{\sigma}_n\eta_0^nt}.
\end{equation}

Noticing that
$$
\lim\limits_{n\to+\infty}\overline{\sigma}_nA_n=A,\quad \lim\limits_{n\to+\infty}\overline{\sigma}_nB_n=B,
\quad \lim\limits_{n\to+\infty}\overline{\sigma}_n\lambda_\alpha^n=\lambda_\alpha,
$$
we have
\begin{equation}\label{3.41}
\lim\limits_{n\to+\infty}\overline{\sigma}_n\eta_0^n=\eta_0,\quad \lim\limits_{n\to+\infty}\widetilde{P}^n\left(\overline{\sigma}_nt,\xi\right)=\widetilde{P}(t,\xi).
\end{equation}

On the other hand, it is shown in \cite{Cannone-Karch-CPAM-2010, Morimoto-KRM-2012}  that $\left(\hat{f}_n(t,\xi), \hat{g}_n(t,\xi)\right)\to
\left(\hat{f}(t,\xi), \hat{g}(t,\xi)\right)$ uniformly as $n\to+\infty$,
 locally in with respect to $(t,\xi)\in\mathbb{R}^+\times\mathbb{R}^3$.
By  \eqref{3.40}, we  obtain from \eqref{3.41} that
\begin{equation}\label{3.42}
\left\|\hat{f}(t,\cdot)-\hat{g}(t,\cdot)
-\widetilde{P}\left(t,\cdot\right)\right\|_{\mathcal{D}^{2+\delta}}\lesssim O(1)e^{-\eta_0t}.
\end{equation}
\eqref{3.42} is exactly \eqref{stability-estimate} and thus the proof of Theorem \ref{Thm1.1} is completed.

\section{Proof of Theorem \ref{Thm1.2}}
\setcounter{equation}{0}
To prove Theorem 1.2, compared with  Theorem \ref{Thm1.1}, we only need to
obtain the uniform $H^N(\mathbb{R}^3)-$estimate \eqref{uniform-H-N-bound} on $f(t,v)$ and the key point  is to deduce the following coercivity estimate.

\begin{Lemma}\label{coercivity-estimate} There exists a sufficiently large positive constant $t_1>0$ such that
\begin{equation}\label{dissipation-f}
\int\limits_{\mathbb{S}^2}\mathcal{B}\left(\frac{\xi\cdot\sigma}{|\xi|}\right)\left(1-\left|\hat{f}(t,\xi^-)\right|\right)d\sigma
\geq \kappa {{} e^{2s\mu_\alpha t}}|\xi|^{2s} \,, \enskip {{}\mbox{if $|\xi| \ge 2$}},
\end{equation}
holds for any $t\geq t_1$ and some positive constant $\kappa>0$ which depends only on $t_1$.
\end{Lemma}
\begin{proof} Notice that
$$
\hat{\Psi}_{\alpha,K}(\xi)\in \mathcal{K}^\alpha(\mathbb{R}^3),\quad \Psi_{\alpha,K}(v)\in L^1(\mathbb{R}^3)\cap H^\infty(\mathbb{R}^3),\quad \int\limits_{\mathbb{R}^3}\Psi_{\alpha,K}(v)dv=1,
$$
we have from Theorem 1.1 of \cite{Morimoto-Wang-Yang-2014} that for $1<\beta<\alpha<2$ that $\Psi_{\alpha,K}(v)\in \mathcal{P}^\beta(\mathbb{R}^3)$, and consequently Lemma 3 of \cite{Alexandre-Desvillettes-Villani-Wennberg-ARMA-2000} shows
that there exists a positive constant $\kappa_1>0$ independent of $t$ and $\xi$ such that
\begin{equation*}
1-\left|\hat{\Psi}_{\alpha,K}(\xi)\right|\geq \kappa_1\min\left\{1, |\xi|^2\right\}.
\end{equation*}
Hence, we have
\begin{equation}\label{dissipation-Psi}
1-\left|\hat{\Psi}_{\alpha,K}\left(e^{\mu_\alpha t}\xi\right)\right|\geq \kappa_1\min\left\{1, \left|e^{\mu_\alpha t}\xi\right|^2\right\},\quad \forall (t,\xi)\in\mathbb{R}^+\times\mathbb{R}^3.
\end{equation}
On the other hand, we have from the $\mathcal{D}^{2+\delta}-$stability estimate
given in Theorem \ref{Thm1.1} that
\begin{eqnarray}\label{conseq-K-stability}
\left|\hat{f}(t,\xi)-\hat{\Psi}_{\alpha,K}\left(e^{\mu_\alpha t}\xi\right)\right|&\leq&O(1)|\xi|^{2+\delta}e^{-\eta_0 t}+\left|\widetilde{P}(t,\xi)\right|\nonumber\\
&\leq&\kappa_2\left(|\xi|^{2+\delta}+|\xi|^2\right)e^{-\eta_0 t}
\end{eqnarray}
holds for any $(t,\xi)\in\mathbb{R}^+\times\mathbb{R}^3$ with
a constant $\kappa_2>0$ independent of
 $t$ and $\xi$.

A direct consequence of \eqref{dissipation-Psi} and \eqref{conseq-K-stability} is
\begin{eqnarray}\label{dissipation-f-2}
1-\left|\hat{f}(t,\xi)\right|&\geq& \left(1-\left|\hat{\Psi}_{\alpha,K}\left(e^{\mu_\alpha t}\xi\right)\right|\right)-\left|\hat{f}(t,\xi)-\hat{\Psi}_{\alpha,K}\left(e^{\mu_\alpha t}\xi\right)\right|\nonumber\\
&\geq& \kappa_1 \min\left\{1, \left|e^{\mu_\alpha t}\xi\right|^2\right\}
-\kappa_2\left(|\xi|^{2+\delta}+|\xi|^2\right)e^{-\eta_0 t},\quad \forall (t,\xi)\in\mathbb{R}^+\times\mathbb{R}^3.
\end{eqnarray}
Thus
\begin{equation}\label{dissipation-f-3}
1-\left|\hat{f}(t,\xi)\right|\geq\max\left\{0, \kappa_1 \min\left\{1, \left|e^{\mu_\alpha t}\xi\right|^2\right\}
-\kappa_2\left(|\xi|^{2+\delta}+|\xi|^2\right)e^{-\eta_0 t}\right\}.
\end{equation}

With \eqref{dissipation-f-3}, we now turn to prove \eqref{dissipation-f}. Firstly, note that
$|\xi^-|^2=|\xi|^2\sin^2\frac\theta 2$. If we choose $t_1>0$ sufficiently large such that
\begin{equation}\label{choose-t-1}
\kappa_1e^{2\mu_\alpha t}-2\kappa_2e^{-\eta_0 t}\geq\frac{\kappa_1}{2}e^{2\mu_\alpha t}
+ \frac{\kappa_1}{2}e^{2\mu_\alpha t_1}-2\kappa_2e^{-\eta_0 t_1}\geq \frac{\kappa_1}{2}e^{2\mu_\alpha t}
,\quad \forall t\geq t_1,
\end{equation}
then for $t\geq t_1$, $|\xi| \ge 2$ and $\theta$ sufficiently small such that
\begin{equation}\label{range-theta}
\theta\in\left[0,\frac{2}{e^{\mu_\alpha t}|\xi|}\right]\subset\left[0,\frac \pi 2\right),
\end{equation}
we have
\begin{eqnarray}\label{lower-bound-f}
&&\kappa_1 \min\left\{1, \left|e^{\mu_\alpha t}\xi^-\right|^2\right\}
-\kappa_2\left(\left|\xi^-\right|^{2+\delta}+\left|\xi^-\right|^2\right)e^{-\eta_0 t}\nonumber\\
&\geq& \kappa_1 \left|e^{\mu_\alpha t}\xi^-\right|^2
-\kappa_2\left(\left|\xi^-\right|^{2+\delta}+\left|\xi^-\right|^2\right)e^{-\eta_0 t}\\
&\geq& \left(\kappa_1 e^{2\mu_\alpha t}
-2\kappa_2e^{-\eta_0 t}\right)\left|\xi^-\right|^2 \ge \frac{\kappa_1}{2}e^{2\mu_\alpha t}\left|\xi^-\right|^2.\nonumber
\end{eqnarray}
Thus for the case when $t\geq t_1$, $|\xi| \ge 2$ and $\theta$ satisfies \eqref{range-theta}, one can deduce from the assumption \eqref{cross-section}, the estimates \eqref{dissipation-f-3}, \eqref{choose-t-1} and \eqref{lower-bound-f} that
\begin{eqnarray*}
&&\int\limits_{\mathbb{S}^2}\mathcal{B}\left(\frac{\xi\cdot\sigma}{|\xi|}\right)\left(1-\left|\hat{f}(t,\xi^-)\right|\right)d\sigma\\
&\geq& 2\pi \int\limits^{\frac {2}{e^{\mu_\alpha t}|\xi|}}_0\mathcal{B}(\cos\theta)  \left(\kappa_1 e^{2\mu_\alpha t}
-2\kappa_2e^{-\eta_0 t}\right)\left|\xi^-\right|^2 \sin\theta d\theta\\
&\ge & \pi \left(\kappa_1 e^{2\mu_\alpha t}
\right)|\xi|^2\int\limits^{\frac {2}{e^{\mu_\alpha t}|\xi|}}_0\mathcal{B}(\cos\theta)\sin^2\frac\theta 2\sin\theta d\theta\\
&\geq & \frac{2 \kappa_1 e^{2\mu_\alpha t}
}{\pi^2}|\xi|^2\int\limits^{\frac {2}{e^{\mu_\alpha t}|\xi|}}_0\mathcal{B}(\cos\theta)\theta^3 d\theta\\
&\geq& \frac{b_0 \kappa_1 e^{2\mu_\alpha t}
}{\pi^2}|\xi|^2 \int\limits^{\frac {2}{e^{\mu_\alpha t}|\xi|}}_0\theta^{1-2s} d\theta
=\frac{2^{1-2s}b_0\kappa_1e^{2s\mu_\alpha t}}{(1-s)\pi^2}|\xi|^{2s}.
\end{eqnarray*}
Here we have used the fact that $\sin\theta\geq \frac {2\theta} {\pi}$ for $0\leq\theta\leq \frac \pi 2$.
This completes the proof of
the lemma.
\end{proof}

With Lemma \ref{coercivity-estimate}, we now deduce the  uniform estimate on $f(t,v)$. Let $\varphi(t,\xi)$ be the Fourier transform of $f(t,v)$ with respect to $v$.
For any $N\in {\mathbb{N}}$, let $M(\xi)=\widetilde{M}\left(|\xi|^2\right)\left(1-X\left(\frac{|\xi|^2}{4}\right)\right)$ with $\widetilde{M}(t)=t^N$ and $X(t)$ defined as in Theorem \ref{Thm1.1}.  Multiplying \eqref{Reformulated-equation} by $2M^2(\xi)\overline{\varphi}(t,\xi)$ with $\overline{\varphi}(t,\xi)$ being the complex conjugate of $\varphi(t,\xi)$ gives
\begin{eqnarray}
&&\frac{d}{dt}\left(\int\limits_{{\mathbb{R}^3}}\left|M(\xi)\varphi(t,\xi)\right|^2d\xi\right)\label{4.9}\\
&=&2\int\limits_{{\mathbb{R}^3}}\int\limits_{{\mathbb{S}^2}}{\mathcal{B}}\left(\frac{\xi\cdot\sigma}{|\xi|}\right)
\Re\Big\{\left(\varphi(t,\xi^+)\varphi(t,\xi^-)
-\varphi(t,\xi)\right)M^2(\xi)\overline{\varphi}(t,\xi)\Big\}d\sigma d\xi\nonumber\\
&=&\underbrace{-\int\limits_{{\mathbb{R}^3}}\int\limits_{{\mathbb{S}^2}}{\mathcal{B}}\left(\frac{\xi\cdot\sigma}{|\xi|}\right)
\Big(\left|M(\xi)\varphi(t,\xi)\right|^2+\left|M(\xi^+)\varphi(t,\xi^+)\right|^2
-2\Re\left\{\varphi(t,\xi^-)\left(M(\xi^+)\varphi(t,\xi^+)\right)\overline{M(\xi)\varphi(t,\xi)}\right\}\Big)d\sigma d\xi}_{J_1}\nonumber\\
&&-\underbrace{\int\limits_{{\mathbb{R}^3}}\int\limits_{{\mathbb{S}^2}}{\mathcal{B}}\left(\frac{\xi\cdot\sigma}{|\xi|}\right)
\Big(\left|M(\xi)\varphi(t,\xi)\right|^2-\left|M(\xi^+)\varphi(t,\xi^+)\right|^2\Big)d\sigma d\xi}_{J_2}\nonumber\\
&&-2\underbrace{\int\limits_{{\mathbb{R}^3}}\int\limits_{{\mathbb{S}^2}}{\mathcal{B}}\left(\frac{\xi\cdot\sigma}{|\xi|}\right)
\Re\Big\{\varphi(t,\xi^-)\left(M(\xi)-M(\xi^+)\right)\varphi(t,\xi^+)\overline{M(\xi)\varphi(t,\xi)}\Big\}d\sigma d\xi}_{J_3}.\nonumber
\end{eqnarray}
We estimate $J_1,\ J_2 $ and $J_3$ term by term
as follows.  Since ${\textrm{Supp}}\ M(\xi)\subset \left\{\xi\in{\mathbb{R}^3}, |\xi|\geq 2\right\}$, it follows
firstly from Lemma \ref{coercivity-estimate} that
\begin{equation}\label{4.10}
J_1\lesssim - e^{2s \mu_\alpha t} \int\limits_{{\mathbb{R}^3}}|\xi|^{2s}\left|M(\xi)\varphi(t,\xi)\right|^2d\xi\,,
%\lesssim -R^{2s}\int\limits_{{\mathbb{R}^3}}\left|M_R(\xi)\varphi(t,\xi)\right|^2d\xi,
\end{equation}
because
\begin{eqnarray*}
&&\left|M(\xi)\varphi(t,\xi)\right|^2+\left|M(\xi^+)\varphi(t,\xi^+)\right|^2
-2\Re\left\{\varphi(t,\xi^-)\left(M(\xi^+)\varphi(t,\xi^+)\right)\overline{M(\xi)\varphi(t,\xi)}\right\}\\
&\geq &(1-|\varphi(t,\xi^-)|)\Big(\left|M(\xi)\varphi(t,\xi)\right|^2+\left|M(\xi^+)\varphi(t,\xi^+)\right|^2\Big)\\
&\geq &(1-|\varphi(t,\xi^-)|)\left|M(\xi)\varphi(t,\xi)\right|^2.
\end{eqnarray*}

For $J_2$, if we use the change of variable $\xi\to\xi^+$ for the term $M(\xi^+)\varphi(t,\xi^+)$, the cancelation lemma (Lemma 1 of \cite{Alexandre-Desvillettes-Villani-Wennberg-ARMA-2000}) implies that
\begin{eqnarray}
\left|J_2\right|&=&2\pi\left|\int\limits_{{\mathbb{R}^3}}\left|M(\xi)\varphi(t,\xi)\right|^2
\left(\int\limits_0^{\frac \pi 2}{\mathcal{B}}(\cos\theta)\sin\theta\left(1-\cos^{-3}\left(\frac\theta 2\right)\right)d\theta\right)d\xi\right|\label{4.11}\\
&\lesssim& \int\limits_{{\mathbb{R}^3}}\left|M(\xi)\varphi(t,\xi)\right|^2d\xi.\nonumber
\end{eqnarray}

For $J_3$, note that
\begin{eqnarray*}
J_3&=& \underbrace{\int\limits_{{\mathbb{R}^3}}\int\limits_{{\mathbb{S}^2}}{\mathcal{B}}\left(\frac{\xi\cdot\sigma}{|\xi|}\right)
\Re\Big\{\varphi(t,\xi^-)\left(\widetilde{M}(\xi)-\widetilde{M}(\xi^+)\right)\left\{1-X\left(\frac{|\xi^+|^2}{4}\right)
\right\}\varphi(t,\xi^+)\overline{M(\xi)\varphi(t,\xi)}\Big\}d\sigma d\xi}_{J^1_3}\\
&&+\underbrace{\int\limits_{{\mathbb{R}^3}}\int\limits_{{\mathbb{S}^2}}{\mathcal{B}}\left(\frac{\xi\cdot\sigma}{|\xi|}\right)
\Re\Big\{\varphi(t,\xi^-)\widetilde{M}(\xi)\left(X\left(\frac{|\xi^+|^2}{4}\right)-X\left(\frac{|\xi|^2}{4}\right)\right)\varphi(t,\xi^+)\overline{M(\xi)\varphi(t,\xi)}\Big\}d\sigma d\xi}_{J^2_3}.
\end{eqnarray*}
We estimate $J^1_3$ and $J^2_3$ separately.

 For $J_3^1$, since $|\xi^+|^2=|\xi|^2\cos^2\frac\theta 2\sim |\xi|^2$ for $\theta\in \left[0,\frac \pi 2\right]$ and $|\xi|^2-|\xi^+|^2=|\xi|^2\sin^2\frac\theta 2$, we have
$$
\left|\widetilde{M}(\xi)-\widetilde{M}(\xi^+)\right|\lesssim \sin^2\left(\frac\theta 2\right)\widetilde{M}(\xi^+),
$$
and consequently
\begin{eqnarray}
\left|J^1_3\right|&\lesssim& \int\limits_{{\mathbb{R}^3}}\left(\int\limits_0^{\frac\pi 2} {\mathcal{B}}(\cos\theta)\sin\theta\sin^2\left(\frac\theta 2\right)d\theta\right)\left|M(\xi)\varphi(t,\xi)\right|\cdot  \left|M(\xi^+)\varphi(t,\xi^+)\right| d\xi\label{4.12}\\
&\lesssim&\int\limits_{{\mathbb{R}^3}}\left|M(\xi)\varphi(t,\xi)\right|^2d\xi.\nonumber
\end{eqnarray}
Here we have used the fact that $|\varphi(t,\xi^-)|\leq 1$.

For $J^2_3$, since
\begin{eqnarray*}
X\left(\frac{|\xi|^2}{4}\right)-X\left(\frac{|\xi^+|^2}{4}\right)
&=&X'\left(\frac{\eta|\xi|^2+(1-\eta)|\xi^+|^2}{4}\right)\frac{|\xi|^2-|\xi^+|^2}{4}\\
&=&\frac{|\xi|^2\sin^2\left(\frac\theta 2\right)}{4}X'\left(\frac{\eta|\xi|^2+(1-\eta)|\xi^+|^2}{4}\right), \quad \eta\in[0,1],
\end{eqnarray*}
and
\begin{eqnarray*}
|\xi^+|^2 \le |\xi|^2\le 2|\xi^+|^2,\quad {\textrm{Supp}}\ \left\{X'\left(\frac{|\xi|^2}{4}\right)\right\}\subset\left\{\xi\in{\mathbb{R}^3}: 4 \leq|\xi|^2\le 8\right\},
\end{eqnarray*}
we obtain
\begin{eqnarray*}
{\textrm{Supp}}\ \left\{\widetilde{M}(\xi)\left(X\left(\frac{|\xi|^2}{4}\right)-X\left(\frac{|\xi^+|^2}{4}\right)\right)\right\}
\subset \left\{\xi\in{\mathbb{R}^3}: 4\leq|\xi|^2\leq 16 \right\}\,.
\end{eqnarray*}
Hence, there exists a constant $C_N >0$ depending on $N$ such that
\begin{eqnarray}
\left|J^2_3\right|&\le& 4^{2N} \int\limits_{|\xi|\leq 4}\left(\int\limits_0^{\frac\pi 2} {\mathcal{B}}(\cos\theta)\sin\theta\sin^2\left(\frac\theta 2\right)d\theta\right)|\varphi(t,\xi)|\cdot|\varphi(t,\xi^+)|d\xi\label{4.13}\\
&\le& C_N  \nonumber
%\int\limits_{{\mathbb{R}^3}}\left|M_R(\xi)\varphi(t,\xi)\right|^2d\xi.
\end{eqnarray}
because  of $|\varphi(t,\xi)|\leq 1$.
\eqref{4.12} together with \eqref{4.13} shows that there exists a $C_1 >0$  such that
\begin{equation}\label{4.14}
\left|J_3\right|\le  C_1  \int\limits_{{\mathbb{R}^3}}|\xi|^{2s}\left|M(\xi)\varphi(t,\xi)\right|^2d\xi + C_N\,.
\end{equation}

Inserting \eqref{4.10}, \eqref{4.11} and \eqref{4.14} into \eqref{4.9}, we
have, for another $C_N' >0$,  %by choosing $R>0$ sufficiently large that
\begin{equation*}%\label{4.15}
\frac{d}{dt}\left(\int\limits_{{\mathbb{R}^3}}\left|M(\xi)\varphi(t,\xi)\right|^2d\xi\right)+
\int\limits_{{\mathbb{R}^3}}\left|M(\xi)\varphi(t,\xi)\right|^2d\xi\leq C'_N\,,
\end{equation*}
which gives
%With \eqref{4.15}, by Theorem 1.3 of \cite{Morimoto-Yang-2012}, the Gronwall inequality yields that for any given sufficiently small time $t_1>0$ and
 %$N\in{\mathbb{N}}$, there exists a positive constant $C(t_1,N)>0$ such that
\begin{equation}\label{4.16}
\int\limits_{{\mathbb{R}^3}}\left|M(\xi)\varphi(t,\xi)\right|^2d\xi
%\int\limits^t_{t_1}\int\limits_{{\mathbb{R}^3}}\left(1+|\xi|^{2}\right)^{N+s}\left|\varphi(\tau,\xi)\right|^2d\xi d\tau
\leq e^{-(t-t_1)} \int\limits_{{\mathbb{R}^3}}\left|M(\xi)\varphi(t_1,\xi)\right|^2d\xi +
C_N', \quad t\geq t_1.
\end{equation}
Noting $|\varphi(\xi) | \le 1$ again, by means of
\eqref{4.16} we see that for any $N \in {\mathbb{N}}$ there exists a $C(t_1,N) >0$ such that
\begin{equation*}%\label{4.17}
\sup\limits_{t\in[t_1,\infty)}\Big\{\|f(t)\|_{H^N}\Big\}\leq C(t_1,N)<+\infty,\quad \forall N\in{\mathbb{N}}.
\end{equation*}
This and \eqref{decay-similarity} give
\begin{equation}\label{4.18}
\sup\limits_{t\in[t_1,\infty)}\Big\{\left\|f(t,\cdot)-f_{\alpha,K}(t,\cdot)\right\|_{H^N}\Big\}\leq C(t_1,N)<+\infty,\quad \forall N\in{\mathbb{N}},\quad t\geq t_1.
\end{equation}
Moreover,
\begin{eqnarray}\label{4.19}
&&\left|\hat{f}(t,\xi)-\hat{\Psi}_{\alpha,K}\left(e^{\mu_\alpha t}\xi\right)\right|^2\nonumber\\
&\lesssim& e^{-\eta_0 t}|\xi|^{2+\delta}\left|\hat{f}(t,\xi)-\hat{\Psi}_{\alpha,K}\left(e^{\mu_\alpha t}\xi\right)-\widetilde{P}(t,\xi)\right|+
\left|\widetilde{P}(t,\xi)\right|^2\\
&\lesssim&e^{-\eta_0 t}|\xi|^{2+\delta}\left|\hat{f}(t,\xi)-\hat{\Psi}_{\alpha,K}\left(e^{\mu_\alpha t}\xi\right)\right|+
e^{-At}|\xi|^4X(\xi)\left(|\xi|^\delta e^{-\eta_0 t}+e^{-At}\right).\nonumber
\end{eqnarray}
\eqref{4.18} and \eqref{4.19}  yield
\begin{eqnarray}
&&\left\|f(t,\cdot)-f_{\alpha,K}(t,\cdot)\right\|^2_{H^N}\nonumber\\
&=&\int\limits_{\mathbb{R}^3}(1+|\xi|^2)^N
\left|\hat{f}(t,\xi)-\hat{\Psi}_{\alpha,K}\left(e^{\mu_\alpha t}\xi\right)\right|^2d\xi\nonumber\\
&\lesssim&e^{-\eta_0 t}\int\limits_{\mathbb{R}^3}(1+|\xi|^2)^N|\xi|^{2+\delta}
\left|\hat{f}(t,\xi)-\hat{\Psi}_{\alpha,K}\left(e^{\mu_\alpha t}\xi\right)\right|d\xi\label{4.20}\\
&&+e^{-At}\int\limits_{\mathbb{R}^3}X(\xi)\left(e^{-\eta_0 t}+e^{-At}\right)d\xi
\lesssim e^{-\eta_0 t}.\nonumber
\end{eqnarray}
Here we have used the fact that
\begin{eqnarray*}
&&\int\limits_{\mathbb{R}^3}(1+|\xi|^2)^N|\xi|^{2+\delta}
\left|\hat{f}(t,\xi)-\hat{\Psi}_{\alpha,K}\left(e^{\mu_\alpha t}\xi\right)\right|d\xi\\
&\leq &\left(\int\limits_{\mathbb{R}^3}(1+|\xi|^2)^{2(N+1)}|\xi|^{2(2+\delta)}
\left|\hat{f}(t,\xi)-\hat{\Psi}_{\alpha,K}\left(e^{\mu_\alpha t}\xi\right)\right|^2d\xi\right)^{\frac 12}\left(\int\limits_{\mathbb{R}^3}(1+|\xi|^2)^{-2}d\xi\right)^{\frac 12}\\
&\leq& C(t_1,N), \quad t\geq t_1.
\end{eqnarray*}
\eqref{4.20} is exactly \eqref{H-N-decay} and the proof of Theorem \ref{Thm1.2} is completed.

\section{Proof of Corollary \ref{decay-Maxwellian}}
\setcounter{equation}{0}

We prove Corollary \ref{decay-Maxwellian} in this last section. Firstly of
all, note that
 Theorem \ref{Thm1.1} and Theorem \ref{Thm1.2} hold for $\alpha=2$.
The purpose of Corollary \ref{decay-Maxwellian} is to have a better
convergence  rate in the case of finite energy.

In fact, compared with  Theorem \ref{Thm1.1} and Theorem \ref{Thm1.2}, the main difference is that now the initial data $f_0(v)$ is of finite energy and consequently the corresponding global solution $F_n(t,v)$ of the Cauchy problem \eqref{3.3}-\eqref{3.4} with $H_0(v)=f_{0R}(v)$  also has finite energy, i.e.
\begin{equation}\label{5.1}
\int\limits_{\mathbb{R}^3}|v|^2F^R_n(t,v)dv=\int\limits_{\mathbb{R}^3}|v|^2f_{0R}(v)dv
\lesssim 1+\int\limits_{\mathbb{R}^3}|v|^2f_0(v)dv<+\infty.
\end{equation}
With \eqref{5.1}, it is straightforward to show that
\begin{equation}\label{5.2}
\left|\hat{F}_n(t,\xi)-1\right|\leq O(1)|\xi|^2,\quad \left|\mu(\xi)-1\right|\leq O(1)|\xi|^2.
\end{equation}
Consequently, the term $I_7$  in \eqref{3.17} can be estimated  by
\begin{equation}\label{5.3}
|I_7|\leq
\left\{
\begin{array}{rl}
O(1)|\xi|^4e^{-A_nt},&\quad |\xi|\leq 1,\\[2mm]
O(1)e^{-A_nt},&\quad |\xi|\geq 1, \quad t\in\mathbb{R}^+.
\end{array}
\right.
\end{equation}
 Here,  $G_n(t,v)=\mu(v)$.

Having  \eqref{5.3}, the proof of
Corollary \ref{decay-Maxwellian} is the same as
the ones for Theorem \ref{Thm1.1} and Theorem \ref{Thm1.2}. Thus, we
 omit the detail for brevity.

\vspace{1cm}

\begin{center}
{\bf Acknowledgment}
\end{center}
The research of the first author was supported  by Grant-in-Aid for Scientific Research No. 25400160, Japan Society of the Promotion of Science. The research of the second author was supported by the  NSFC-RGC Grant, N-CityU102/12. The research of the third author was supported  by the Fundamental Research Funds for the Central Universities of China and three grants from the National Natural Science Foundation of China under contracts 10925103, 11271160 and 11261160485.


\vskip 1cm \small
\begin{thebibliography}{99}

\bibitem{Alexandre-Desvillettes-Villani-Wennberg-ARMA-2000} R. Alexandre, L. Desvillettes, C. Villani, and B. Wennberg, Entropy dissipation and longrange interactions. {\it Arch. Rational Mech. Anal.} {\bf 152} (2000), 327-355.

\bibitem{Arkeryd-CMP-1982} L. Arkeryd, Asymptotic behaviour of the Boltzmann equation with infinite range forces. {\it Comm. Math. Phys.} {\bf 86} (1982), 475-484.

\bibitem{Bobylev-DANS-1975}A. V. Bobylev, The method of the Fourier transform in the theory of the Boltzmann equation for
Maxwell molecules. {\it Dokl. Akad. Nauk SSSR} {\bf 225} (6) (1975), 1041-1044.

\bibitem{Bobylev-1988} A. V. Bobylev, The theory of the nonlinear spatially uniform Boltzmann equation for Maxwell
molecules. In {\it Mathematical physics reviews, Vol. 7,} volume 7 of {\it Soviet Sci. Rev. Sect. C Math. Phys.
Rev.}, pages 111¨C233. Harwood Academic Publ., Chur, 1988.

\bibitem{Bobylev-Cercignani-JSP-2002a} A. V. Bobylev and C. Cercignani, Exact eternal solutions of the Boltzmann equation. {\it J. Statist. Phys.} {\bf 106} (5-6) (2002), 1019-1038.

\bibitem{Bobylev-Cercignani-JSP-2002b} A. V. Bobylev and C. Cercignani, Self-similar solutions of the Boltzmann equation and their applications. {\it J. Statist. Phys.} {\bf 106} (5-6) (2002), 1039-1071.

\bibitem{Cannone-Karch-CPAM-2010} M. Cannone and G. Karch, Infinite energy solutions to the homogeneous Boltzmann equation. {\it Comm. Pure Appl. Math.} {\bf 63} (6) (2010), 747-778.

\bibitem{Cannone-Karch-KRM-2013} M. Cannone and G. Karch, On self-similar solutions to the homogeneous Boltzmann equation. {\it Kinetic and Related Models} {\bf 6} (4) (2013), 801-808.

%\bibitem{Desvillettes-Villani-InventMath-2005} L. Desvillettes and C. Villani, On the trend to global equilibrium for spatially inhomogeneous kinetic systems: The Boltzmann equation. {\it Invent. Math.}  159 (2005), 245-316.

\bibitem{Gabetta-Toscani-Wennberg-JSP-1995} E. Gabetta, G. Toscani, and B. Wennberg, Metrics for probability distributions and the trend to equilibrium for solutions of the Boltzmann equation. {\it J. Stat. Phys.} {\bf 81} (5-6) (1995), 901-934.

\bibitem{Ikenberry-Truesdell-JRMA-1956}E. Ikenberry and C. Truesdell, On the pressure and the flux of energy according to Maxwell's kinetic energy I. {\it J. Rat. Mech. Anal.} {\bf 5} (1956), 1-54.

\bibitem{Morimoto-KRM-2012} Y. Morimoto, A remark on Cannone-Karch solutions to the homogeneous Boltzmann equation for
Maxwellian molecules. {\it Kinet. Relat. Models} {\bf 5} (3) (2012), 551-561.


%\bibitem{Morimoto-Ukai-Xu-Yang-DCDS-2009} Y. Morimoto, S. Ukai, C.-J. Xu, and T. Yang, Regularity of solutions to the spatially homogeneous Boltzmann equation without angular cutoff. {\it Discrete Contin. Dyn. Syst.} {\bf 24} (1) (2009), 187-212.

\bibitem{Morimoto-Wang-Yang-2014} Y. Morimoto, S.-K. Wang, and T. Yang, A new characterization and global regularity of infinite energy solutions to the homogeneous Boltzmann equation. {\it J. Math. Pures Appl. (9)} {\bf 103}  (2015), no. 3, 809-829.

\bibitem{Morimoto-Yang-2012} Y. Morimoto and T. Yang, Villani conjecture on smoothing effect of the homogeneous Boltzmann
equation with measure initial datum. {\it Ann. Inst. H. Poincare Anal. Non Lineaire}  {\it 32} (2015), no. 2, 429-442.

%\bibitem{Mouhot-CPDE-2005} C. Mouhot, Quantitative lower bounds for the full Boltzmann equation, Part I: Periodic
%boundary conditions. {\it Communications in Partial Differential Equations} {\bf 30} (2005), 881-917.

\bibitem{Pulvirenti-Toscani-AMPA-1996} A. Pulvirenti and G. Toscani, The theory of the nonlinear Boltzmann equation for Maxwell molecules in Fourier representation. {\it Ann. Mat. Pura Appl.} {\bf 171} (1996), 181-204.

\bibitem{Tanaka-WVG-1978} H. Tanaka, Probabilistic treatment of the Boltzmann equation of Maxwellian molecules.
{\it Wahrsch. Verw. Geb.} {\bf 46} (1978), 67-105.

\bibitem{Toscani-Villani-JSP-1999} G. Toscani and C. Villani, Probability metrics and uniqueness of the solution to the Boltz-
mann equations for Maxwell gas. {\it J. Statist. Phys.} {\bf 94} (1999), 619-637.

\bibitem{Villani-ARMA-1998} C. Villani, On a new class of weak solutions to the spatially homogeneous Boltzmann and
Landau equations. {\it Arch. Rational Mech. Anal.} {\bf 143} (1998), 273-307.

\bibitem{Vinnani-2002} C. Villani. A review of mathematical topics in collisional kinetic theory. In {\it Handbook of mathematical fluid dynamics, Vol. I}, pages 71-305. North-Holland, Amsterdam, 2002.

%\bibitem{Wennberg-Rendiconti-1996} B. Wennberg, The Povzner inequality and moments in the Boltzmann equation. {\it Rendiconti del Circolo Matematico di Palermo Serie II Suppl.}
%{\bf 45} (1996), 673-681.

%\bibitem{Toscani-Villani-JSP-1999} G. Toscani and C. Villani, Probability metrics and uniqueness of the solution to the Boltzmann equation for a Maxwell gas. {\it J. Statist. Phys.} {\bf 94} (3-4) (1999), 619-637.

























\end{thebibliography}
\normalsize
\end{document}